\documentclass[10pt]{amsart}
\usepackage[a4paper,margin=1in]{geometry}
\usepackage{microtype}
\usepackage[titletoc,toc,page]{appendix}
\usepackage{float}
\usepackage{booktabs}
\usepackage{caption}

\usepackage{mathtools, amsthm, amssymb, amsfonts, mathrsfs, graphicx, paralist, lscape, bm, enumitem, amsmath, stmaryrd}
\usepackage{tikz, tikz-cd}
\usetikzlibrary{cd}
\usetikzlibrary{matrix,arrows}
\usepackage[colorlinks=true, pdfstartview=FitV, linkcolor=black, citecolor=black, urlcolor=black]{hyperref}
\usepackage[font=itshape]{quoting}
\numberwithin{equation}{section}
\newtheorem{thm}{Theorem}[section]
\newtheorem*{thm*}{Theorem}
\newtheorem{prop}[thm]{Proposition}
\newtheorem*{prop*}{Proposition}

\newtheorem{lemma}[thm]{Lemma}
\newtheorem*{lemma*}{Lemma}
\newtheorem{cor}[thm]{Corollary}
\newtheorem*{cor*}{Corollary}

\theoremstyle{definition}
\newtheorem{defn}[thm]{Definition}
\newtheorem{example}[thm]{Example}

\newtheorem{examples}[thm]{Examples}

\newtheorem*{eg*}{Example}
\newtheorem*{egs*}{Examples}
\newtheorem{remark}[thm]{Remark}
\newenvironment{rmk}[1][]{\begin{remark}[#1] \pushQED{\qed}}{\popQED \end{remark}}
\newtheorem{remarks}[thm]{Remarks}

\newtheorem*{defn*}{Definition}

\theoremstyle{remark}
\newtheorem*{rmk*}{Remark}

\usepackage{newtxtext}

\usepackage{fancyhdr}
\usepackage[
    backend=biber,
    maxbibnames=99,
    maxcitenames=99,
    minbibnames=99,
    mincitenames=99
]{biblatex}
\allowdisplaybreaks

\title{Transversality Methods for Homotopy Groups of Stable Loci in Affine GIT Quotients}
\author{Yizhi Wang}
\date{}
\address{Department of Mathematics, University of York, York, YO10 5GH, United Kingdom}
\email{cfc536@york.ac.uk}
\begin{document}

\begin{abstract}
    We investigate the homotopy groups of stable loci in affine Geometric Invariant Theory (GIT), arising from linear actions of complex reductive algebraic groups on complex affine spaces. Our approach extends the infinite-dimensional transversality framework of Daskalopoulos–Uhlenbeck and Wilkin to this general GIT setting. Central to our method is the construction of a G-equivariant holomorphic vector bundle over the conjugation orbit of a one-parameter subgroup (1-PS), whose fibres are precisely the negative weight spaces determining instability.

    A key proposition establishes that a naturally defined evaluation map is transverse to the zero section of this bundle, implying that generic homotopies avoid all unstable and strictly semistable strata under certain dimensional inequalities. Consequently, our main theorem shows that the stable locus $V^{st}(\rho)$ is $(d_{\text{min}} - 2)$-connected. The connectivity bound is defined as $d_{\text{min}} = \min_{j} (2m_j - 2 \dim_{\mathbb{C}}(\mathrm{G} \cdot \lambda_j))$, where $m_j$ are the ranks of the negative weight spaces and $\dim_{\mathbb{C}}(\mathrm{G} \cdot \lambda_j)$ are the dimensions of the relevant 1-PS orbits.

    Our result also covers cases where semistability does not coincide with stability. The applicability of this framework is illustrated on several examples. In linear control theory, where GIT stability corresponds to the notion of controllability, our results determine the connectivity of the space of controllable systems. In statistical modelling, we compute the connectivity bounds for the space of samples in star-shaped Gaussian graphical models for which the maximum likelihood estimate is unique. We also consider Helmke systems and show that for stability parameters satisfying certain bounds, the space of systems that are both controllable and observable is exactly the space of stable points. The main result can then be used to compute the connectivity of this space.
\end{abstract}

\maketitle

\section{Introduction}
Geometric Invariant Theory (GIT) provides a powerful framework for constructing moduli spaces in terms of quotients by group actions \cite{MumfordFogartyKirwan1994}. This paper utilises GIT techniques to investigate the homotopy groups of stable loci arising from linear actions of complex reductive algebraic groups on complex affine spaces.

An important example of a GIT quotient is the moduli space of semistable holomorphic bundles over a compact Riemann surface. When the degree and rank of the bundle are not coprime, the stable locus is an open subset of this moduli space, and it is an important problem to understand the topology of this subset. Daskalopoulos and Uhlenbeck \cite{DaskalopoulosUhlenbeck1995} used an infinite-dimensional version of Thom's transversality theorem to compute the homotopy groups of the stable locus up to a dimensional bound that can be computed from the degree and rank of the bundle. Their method demonstrated that the space of stable holomorphic structures is $2(g-1)(n-1)$-connected, where $g\geq 2$ is the genus of the Riemann surface.

Let $\mathrm{G}$ be a complex reductive algebraic group acting linearly on a complex affine space $V \cong \mathbb{C}^n$ via a linear representation $\sigma : \mathrm{G} \to \operatorname{GL}(V)$. Given a character $\rho : \mathrm{G} \to \mathbb{C}^\times$, one can define stability conditions on $V$ using semi-invariant polynomials transforming under the action of $\mathrm{G}$ with weights determined by $\rho$. $V$ can then be stratified into $\rho$-stable, $\rho$-semistable, and $\rho$-unstable loci. Hoskins \cite{HoskinsStratAffine}
showed that the Hesselink's stratification by adapted one-parameter subgroups (1-PS) coincides with the Morse-theoretic stratification induced by the norm square of the moment map associated with a maximal compact subgroup $K\subset \mathrm{G}$.

Wilkin \cite{WilkinQuiverTransversality} adapted these techniques to compute homotopy groups of the stable locus of moduli spaces of representations of a quiver, where the space of $\alpha$-stable representations is shown to have trivial homotopy group up to a dimension determined by the minimum of the unstable strata and a dimensional bound determined from the strictly semistable locus. In this paper, we generalise Wilkin's method to the case of affine GIT quotients associated with linear representations.

Since this paper studies homotopy groups of stable loci in affine GIT, a natural question is how its fundamental group relates to that of the original space. In the projective setting, Biswas, Hogadi, and Parameswaran \cite{biswas2014fundamentalgroupgeometricinvariant} prove that if a connected reductive group $\mathrm{G}$ acts on a smooth connected projective variety $M$ and the GIT quotient $M\sslash \mathrm{G}$ is nonempty, then the induced homomorphism $\pi_1(M) \to \pi_1(M\sslash \mathrm{G})$ on fundamental groups is an isomorphism. In particular, over $\mathbb{C}$, the corresponding map of topological fundamental groups is also an isomorphism \cite[Theorem 4.4]{biswas2014fundamentalgroupgeometricinvariant}. This phenomenon has a symplectic analogue. For Hamiltonian actions of connected compact Lie groups, Li \cite{Li2006} shows that the fundamental group of a compact symplectic manifold identifies with that of its symplectic reduction. Moreover, in the projective setting, the Kempf-Ness theorem \cite[Theorem 8.3]{MumfordFogartyKirwan1994} identifies the projective GIT quotient with the symplectic reduction by a maximal compact subgroup, which allows one to consider the $\pi_1$-invariance results for projective GIT in the symplectic context.

By contrast, the present paper studies the homotopy groups of the stable locus $V^{st}$ inside an affine representation space. While Biswas et al. \cite{biswas2014fundamentalgroupgeometricinvariant} show that, in the projective setting, the fundamental group of a GIT quotient is the same as that of the original space, our main result Theorem \ref{Thm 2.6} gives a connectivity result for the stable locus itself. More precisely, if
$$
d_{\min}= \min_j\{ 2\dim_\mathbb{C}V(\lambda_j)_- - 2\dim_\mathbb{C} (\mathrm{G}\cdot \lambda_j)\},
$$
where $V(\lambda_j)_-$ denotes the negative weight space, then the stable locus $V^{st}$ is $(d_{\min}-2)$-connected. This result is equivalent to
$$
\pi_n(V^{st}) = 0 \quad \text{for } n + 1 < d_{\min}.
$$
Thus the present paper computes homotopy groups of the stable locus inside the semistable locus up to an explicit dimensional bound. A further point of context is provided by Hoskins \cite{HoskinsStratAffine}, where she proves an affine version of the Kempf-Ness theorem with respect to a character $\rho$. In the present paper, however, the ambient space is $V \cong \mathbb{C}^N$, which is contractible, and the nontriviality of homotopy groups arises from removing unstable and strictly semistable strata.

In Section 2, we construct a $\mathrm{G}$-equivariant holomorphic vector bundle $W \to \mathrm{G}\cdot \lambda$ over the conjugation orbit of a destabilising one-parameter subgroup. The fibres of this bundle are the negative weight spaces $W_{\lambda'} \cong V(\lambda')_-$, given by the Hilbert-Mumford criterion. A point is destabilised by $\lambda'$ when its component in the fibre $V(\lambda')_-$ vanishes, corresponding to the zero section $\mathcal{O}_W$ of the bundle. We define a map $D$ from the space of homotopies into $W$, which detects these negative weight components. Proposition \ref{Prop 2.4} establishes that $D$ is transverse to the zero section $\mathcal{O}_W$. This transversality condition implies that, for generic homotopies, the preimage of the zero section is a submanifold of the expected codimension. In particular, under the relevant dimension inequality, this preimage is empty, so the homotopy lies entirely in the stable locus.

As in previous works by Daskalopoulos \& Uhlenbeck and Wilkin, this transversality condition implies that, for generic homotopies $f \in \mathcal{F}$ satisfying the dimension inequality $n+1+\dim_\mathbb{R}(\mathrm{G}\cdot \lambda) < 2 \cdot \operatorname{rank}_\mathbb{C}(W)$, the image of $f$ avoids the non-stable strata associated with that conjugacy class of destabilising 1-PSs. Consequently, the main theorem Theorem \ref{Thm 2.6} proves that the stable locus $V^{st}(\rho)$ is $(d_{\min} - 2)$-connected, where $d_{\min}$ is defined in terms of the minimal rank of the negative weight spaces. This result is useful when semistability does not coincide with stability, and the transversality techniques employed in this paper thus provide a method for determining the connectivity and calculating the homotopy groups of stable loci in general affine GIT settings.

An example of such a setting is provided in Section \ref{Helmke systems Section}, which describes Helmke systems in control theory \cite{Bader2008}. Proposition \ref{stability and  controllability and observability} shows that (for certain stability parameters) the stable points correspond precisely to the systems that are both controllable and observable. Theorem \ref{Thm 2.6} can then be used to compute the connectivity of the space of controllable and observable systems.

Section 3 contains two applications of the general theorem. First, we consider classical linear control systems
$$\dot{x} = Ax + Bu,
$$
where $\operatorname{GL}_n(\mathbb{C})$ acts on the space of pairs $(A,B)$ by change of basis in the state space. For the determinant character, GIT stability is equivalent to controllability \cite{Geiss2006ModuliQuivers}. Applying our transversality results gives connectivity properties of the space of controllable systems. The resulting path-connectedness, for instance, implies that any controllable system $(A, B)$ can be continuously deformed into any other controllable system $(A', B')$ through a path consisting entirely of controllable systems.

The second application we consider is star-shaped Directed Acyclic Graph (DAG) models, which correspond to a linear regression problem. In this setting, GIT stability corresponds to the existence of a unique the Maximum Likelihood Estimate (MLE). Applying our transversality results allows us to determine the connectivity of the space $V^{st}$ of samples with a unique MLE \cite{bérczi2023completecollineationsmaximumlikelihood}. For instance, the path-connectedness implies that any sample $f$ with a unique MLE can be continuously deformed into any other such sample through a continuous path that lies entirely in $V^{st}$. This demonstrates that the property of having a unique MLE is robust to small perturbations in the data and that the space of well-behaved data forms a single, connected space, which thus provides a foundation for methods that resolve the non-identifiability of degenerate data by perturbing the data towards this stable locus.

Section 4 further extends the application of the transversality method to Helmke systems, which generalise a linear dynamical system $(A, B, C, D)$ by introducing additional matrices $E$ and $F$, so that the system is described by the equations
$$
Ex_{k+1} = Ax_k + Bu_k, \quad Fy_k = Cx_k + Du_k.
$$
Bader \cite{Bader2008} showed that the space $\mathcal{H}_{n, m, p}$ of Helmke systems can be realised as a quiver representation space with marked vertices, equipped with an action of $\mathrm{GL}_n(\mathbb{C}) \times \mathrm{GL}_n(\mathbb{C}) \times \mathrm{GL}_p(\mathbb{C})$ by change of bases. For characters lying in a suitable chamber, Helmke controllability is equivalent to GIT stability and semistability. This construction is useful because the moduli space of controllable classical linear dynamical system is non-projective, whereas the Helmke systems give a smooth projective compactification containing the original moduli space as a dense open subset.

We study the case of Helmke systems of type $(2, 3, 1)$ and the boundary character $\rho = (-2, 3, 1)$. For this character, semistability and stability need not coincide. We construct explicit stable and strictly semistable Helmke systems, analyse the corresponding subrepresentation inequalities, and compute the relevant negative weight dimensions and orbit dimensions. This gives $d_{\min} = 2$, and hence the non-empty stable locus is $\mathcal{H}^{st}_{2, 3, 1}(\rho)$ is path-connected. We also study the relation between Helmke-observability and GIT stability, comparing the resulting character chambers with the classical controllability and observability chambers.

The paper is organised as follows. In Section 2, we lay out the general theoretical framework, introducing the necessary GIT background, and proving the key transversality result and its implications for homotopy groups. Section 3 contains the applications to controllable linear systems and to star-shaped Gaussian graphical models. In Section 4, we consider Helmke systems, where we recall Bader's quiver description \cite{Bader2008}, analyse a boundary character for systems of type $(2, 3, 1)$, and apply our transversality framework to obtain a connectivity result for systems of this type. Moreover, we study the relation between Helmke-observability and GIT stability.

\section{Geometric Invariant Theory and Transversality}\label{GITANDTRANSECTION}
Let $\mathrm{G}$ be a complex reductive algebraic group acting linearly on a complex affine space $V \cong \mathbb{C}^N$ for some $N = \operatorname{dim}_\mathbb{C}V$, via a faithful linear representation $\sigma : \mathrm{G} \hookrightarrow \operatorname{GL}(V)$. The $\mathrm{G}$-action on $V$ induces an action on its ring of regular functions, $\mathbb{C}[V]$, defined by $(g\cdot f)(v) = f(g^{-1} \cdot v)$ for $g\in \mathrm{G}, v \in V, f \in \mathbb{C}[V]$. We denote the subring of $\mathrm{G}$-invariant functions by $\mathbb{C}[V]^\mathrm{G}$, and then the affine GIT quotient is defined as $V/\!/\mathrm{G} := \operatorname{Spec}\mathbb{C}[V]^\mathrm{G}$. Note that $\mathbb{C}[V]^\mathrm{G}$ being finitely generated $\mathbb{C}$-algebra \cite{Dolgachev_2003} ensures that the quotient $V/\!/\mathrm{G}$ is an affine variety.

Let $\rho : \mathrm{G} \to \mathbb{C}^\times$ be a character of $\mathrm{G}$. A regular function $f \in \mathbb{C}[V]$ is a semi-invariant of weight $\rho^n$ if it transforms under the $\mathrm{G}$-action associated to $f(g \cdot v) = \rho(g)^nf(v)$ for all $g\in \mathrm{G}, v \in V$. The set of such functions for a fixed $n$ is denoted $\mathbb{C}[V]^\mathrm{G}_{\rho^n}$, and the collection of all such semi-invariants forms a graded ring $\bigoplus_{n\in\mathbb{Z}}\mathbb{C}[V]^\mathrm{G}_{\rho^n}$.

With a chosen character $\rho$, we can define notions of stability for points in $V$.
\begin{defn}
    Let $v \in V$.
    \begin{itemize}
        \item $v$ is \emph{$\rho$-semistable} if there exists a non-constant $\rho^n$-semi-invariant polynomial $f \in \mathbb{C}[V]^\mathrm{G}_{\rho^n}$ for some integer $n > 0$ such that $f(v) \neq 0$. The set of all semistable points form the $\rho$-semistable locus $V^{ss}(\rho)$.
        \item A $\rho$-semistable point $v$ is \emph{$\rho$-stable} if its $\mathrm{G}$-orbit $\mathrm{G}\cdot v$ is closed in $V^{ss}(\rho)$, and its stabiliser subgroup $\mathrm{G}_v = \{g\in \mathrm{G} \mid g \cdot v = v \}$ is finite. The set of such points is denoted $V^{st}(\rho)$.
        \item $v$ is \emph{$\rho$-unstable} if it is not $\rho$-semistable. This means $f(v)=0$ for all non-constant $f \in \mathbb{C}[V]^\mathrm{G}_{\rho^n}$ with $n > 0$, and we have $V^{us}(\rho) = V\setminus V^{ss}(\rho)$.
    \end{itemize}
\end{defn}
A one-parameter subgroup of $\mathrm{G}$ is a group homomorphism $\lambda: \mathbb{C}^\times \to \mathrm{G}$. The Hilbert-Mumford Criterion provides a numerical way to check these stability conditions by examining the behaviour of points under the action of one-parameter subgroups (1-PSs) of $\mathrm{G}$.

Let $K\subset \mathrm{G}$ be the maximal compact subgroup of $\mathrm{G}$, and $V$ be endowed with a $K$-invariant Hermitian inner product. Fix a non-trivial one-parameter subgroup $\lambda : S^1 \to \mathrm{G}$. The conjugation orbit of $\lambda$ is
$$
\mathrm{G}\cdot \lambda := \{g\lambda g^{-1} \mid g \in \mathrm{G} \} \cong \mathrm{G}/C_\mathrm{G}(\lambda),
$$
where $C_\mathrm{G}(\lambda)$ is the centraliser of $\lambda(S^1)$ in $\mathrm{G}$. In particular, $\mathrm{G} \cdot \lambda$ is a smooth projective variety. For each $\lambda' \in \mathrm{G}\cdot \lambda$, the representation $V$ decomposes into
$$
V = \bigoplus_{i \in \mathbb{Z}}V_i(\lambda'),\quad V_i(\lambda') = \{v\in V \mid \lambda'(t) \cdot v = t^iv \},
$$
and so we can denote
$$
V(\lambda')_- = \bigoplus_{i < 0}V_i(\lambda'),\quad V(\lambda')_{\geq 0} = \bigoplus_{i\geq 0} V_i(\lambda').
$$
Set
\begin{equation}\label{eq: 1}
W := \{(\lambda',v)\in (\mathrm{G}\cdot \lambda) \times V \mid v \in V(\lambda')_- \}. \tag{1}
\end{equation}
The projection $\pi : W \to \mathrm{G}\cdot \lambda$ defined by $(\lambda', v) \mapsto \lambda'$ is then a vector bundle. Then the \emph{zero section} of $W$ is denoted by $\mathcal{O}_W := \{(\lambda',0) \mid \lambda' \in \mathrm{G}\cdot \lambda \}$.

\begin{lemma}
    Let $v \in V$ and $\lambda' \in \mathrm{G} \cdot \lambda$. The point $v$ is destabilised by the 1-PS $\lambda'$ if and only if $(\lambda', \pi^\perp_{\lambda'}(v)) \in \mathcal{O}_W$, where $\pi^\perp_{\lambda'}: V \to V(\lambda')_-$ is the projection onto the negative weight space $V(\lambda')_-$.
\end{lemma}
\begin{proof}
    By the Hilbert-Mumford criterion, a point $v$ is destabilised by a given 1-PS $\lambda'$ if the limit $\lim_{t \to 0}\lambda'(t) \cdot v$ exists in $V$. Note that for any $\lambda' \in \mathrm{G} \cdot \lambda$, we have the weight space decomposition
    $$
    V = \bigoplus_{i \in \mathbb{Z}} V_i(\lambda'),
    $$
    where the action on each subspace is given by $\lambda'(t) \cdot u = t^i u$ for all $u \in V_i(\lambda')$. A point $v \in V$ can be decomposed into a finite sum $v = \sum_{ i \in \mathbb{Z}} v_i$ with $v_i \in V_i(\lambda')$. So the action of the 1-PS on $v$ is
    $$
    \lambda'(t) \cdot v = \sum_{i \in \mathbb{Z}} t^iv_i.
    $$

    For the limit $\lim_{t \to 0} \sum_{i \in \mathbb{Z}}t^iv_i$ to exist in the affine space $V$, the expression cannot contain any terms with negative powers of $t$. This requires that $v_i = 0$ for all $i < 0$. The direct sum of these negative weight spaces forms $V(\lambda)_- = \bigoplus_{i < 0} V_i(\lambda')$. So the orthogonal projection of $v$ onto this negative weight space is $\pi^\perp_{\lambda'}(v) = \sum_{i < 0} v_i$.

    Consequently, the limit exists if and only if $\pi^\perp_{\lambda'} (v) = 0$, which holds if and only if the pair $(\lambda', \pi^\perp_{\lambda'}(v)) = (\lambda', 0)$. This is equivalent to the condition $(\lambda', \pi^\perp_{\lambda'}(v)) \in \mathcal{O}_W$.
\end{proof}

\begin{prop}
    $\pi : W \to \mathrm{G}\cdot \lambda$ is a holomorphic $\mathrm{G}$-equivariant vector bundle of rank $m := \operatorname{dim}_\mathbb{C}V(\lambda)_-$, which is isomorphic to the associated bundle $\mathrm{G} \times_{C_\mathrm{G}(\lambda)} V(\lambda)_- \to \mathrm{G}/C_\mathrm{G}(\lambda)$.
\end{prop}
\begin{proof}
    Consider the map $\Phi : \mathrm{G}\times V(\lambda)_- \to W$ defined by
    $$
    \Phi(g,v) = (g\lambda g^{-1}, g\cdot v).
    $$
    To check that the codomain is $W$, let $\lambda' = g\lambda g^{-1}$, and consider the action of $\lambda'(t)$ on $g \cdot v_i$, which is given by $\lambda'(t)(g\cdot v_i) = (g\lambda(t)g^{-1})(g\cdot v_i) = g\lambda(t)(g^{-1}g)v_i = g\lambda(t)v_i = g(t^iv_i) = t^i(g \cdot v_i)$. Hence for any $(g,v) \in \mathrm{G} \times V(\lambda)_-$, the vector $g\cdot v$ lies in $V(g\lambda g^{-1})_-$, so $\Phi(g,v) \in W$. Moreover, let $n\in C_\mathrm{G}(\lambda)$, then
    $$\Phi(gn^{-1}, n\cdot v ) = (gn^{-1}\lambda ng^{-1}, gn^{-1}n\cdot v) = (g\lambda g^{-1}, g\cdot v) = \Phi(g, v),
    $$
    which implies that $\Phi$ is invariant under the $C_\mathrm{G}(\lambda)$-action.
    
    Note that $\mathrm{G}\times_{C_\mathrm{G}(\lambda)}V(\lambda)_- := (\mathrm{G} \times V(\lambda)_-)/\sim$, where $(g_1,v_1) \sim (g_2,v_2)$ if and only if there exists an $n\in C_\mathrm{G}(\lambda)$ such that $(g_2,v_2) = (g_1n^{-1}, n\cdot v_1)$, therefore $\Phi$ descends to a well-defined map
    $$
    \widetilde{\Phi} : \mathrm{G}\times_{C_\mathrm{G}(\lambda)}V(\lambda)_- \to W, \quad \widetilde{\Phi}([g,v]) = \Phi(g,v)=(g\lambda g^{-1}, g\cdot v),
    $$
    and we shall prove that $\widetilde{\Phi}$ is bijective.
    
    For the surjectivity, let $(\lambda',w) \in W$ be an arbitrarily chosen element, then there exists some $g\in \mathrm{G}$ such that $\lambda' = g\lambda g^{-1}$. Define $v := g^{-1}\cdot w$, then for $v_j = g^{-1}w_j$ we have $\lambda(t)v_j = \lambda(t)(g^{-1}w_j) = g^{-1}(g\lambda(t)g^{-1})w_j = g^{-1}(t^jw_j) = t^j(g^{-1}w_j) = t^jv_j$ Thus $v_j \in V(\lambda)_-$, which implies that $v = g^{-1}w \in V(\lambda)_-$. Now consider $[g,v] \in \mathrm{G}\times_{C_\mathrm{G}(\lambda)}V(\lambda)_-$, observe that $\widetilde{\Phi}([g,v]) = (g\lambda g^{-1}, g\cdot v) = (\lambda', g\cdot(g^{-1}w)) = (\lambda', w)$. Hence $\widetilde{\Phi}$ is surjective.

    To prove the injectivity, assume $\widetilde{\Phi}([g_1,v_1]) = \widetilde{\Phi}([g_2,v_2])$. This means $(g_1\lambda g_1^{-1}, g_1\cdot v_1) = (g_2\lambda g_2^{-1}, g_2\cdot v_2)$, then $g_1\lambda g_1^{-1} = g_2\lambda g_2^{-1}$ implies that $(g_2^{-1}g_1) \lambda (g_2^{-1}g_1)^{-1} = \lambda$. Let $h = g_2^{-1}g_1 \in C_\mathrm{G}(\lambda)$, which gives $g_1 = g_2h$. Furthermore, note that from $g_1\cdot v_1 = g_2\cdot v_2$ we can get $(g_2h)\cdot v_1 = g_2 \cdot v_2$, implying $h \cdot v_1 = v_2$. Thus we have $[g_1, v_1] = [g_2h, v_1]$, which is equivalent to $((g_2h)h^{-1}, h\cdot v_1) = (g_2, h\cdot v_1) = (g_2, v_2)$. Therefore, we have shown that $[g_1, v_1] = [g_2, v_2]$, establishing the injectivity, and we can thus conclude that $\widetilde{\Phi}$ is bijective.

    Now we show the holomorphicity of $\widetilde{\Phi}$. Note that $\mathrm{G}$ is complex Lie group, of which the group multiplication, inversion, and the conjugation are holomorphic maps. And the $\mathrm{G}$-action on $V$ is a holomorphic representation, thus $\Phi$ is a holomorphic map. Since $C_\mathrm{G}(\lambda)$ is a closed complex subgroup of $\mathrm{G}$, the $C_\mathrm{G}(\lambda)$-action is also holomorphic, which is in addition free and proper, hence the quotient $\mathrm{G} \times_{C_\mathrm{G}(\lambda)}V(\lambda)_-$ is a complex manifold and the quotient map $q: \mathrm{G}\times V(\lambda)_- \to \mathrm{G}\times_{C_\mathrm{G}(\lambda)}V(\lambda)_-$ is a holomorphic submersion. Because $\Phi = \widetilde{\Phi}\circ q$, it follows that $\widetilde{\Phi}$ is holomorphic. It is furthermore a biholomorphism by the Inverse Function theorem, where it suffices to show the complex differential $d\widetilde{\Phi}_{[g,v]}$ is a $\mathbb{C}$-linear isomorphism at every point $[g,v] \in \mathrm{G} \times_{C_\mathrm{G}(\lambda)}V(\lambda)_-$, and thus $W$ can be endowed with a complex manifold structure making $\pi: W \to \mathrm{G}\cdot \lambda$ a holomorphic map.

    Since $\widetilde{\Phi}$ is a biholomorphism that is also linear on the fibres and commutes with projections, it is an isomorphism of holomorphic vector bundles. Thus $W$ inherits the structure of a holomorphic vector bundle from $\mathrm{G} \times_{C_\mathrm{G}(\lambda)}V(\lambda)_-$. To see the $\mathrm{G}$-equivariance of $\widetilde{\Phi}$, let $g_0\in \mathrm{G}$ be arbitrary. Then we have $\widetilde{\Phi}(g_0\cdot [g,v]) = ((g_0g)\lambda(g_0g)^{-1}, (g_0g)\cdot v) = (g_0(g\lambda g^{-1})g_0^{-1}, g_0(g\cdot v)) = g_0\cdot (g\lambda g^{-1}, g\cdot v) = g_0 \cdot \widetilde{\Phi}([g,v])$.

    Note that the fibre $W_{\lambda'} = V(\lambda')_-$. If $\lambda' = g\lambda g^{-1}$, the map $v \mapsto g \cdot v$ restricts to an isomorphism from $V(\lambda)_-$ to $V(\lambda')_-$. Hence $\operatorname{dim}_\mathbb{C}V(\lambda')_- = \operatorname{dim}_\mathbb{C}V(\lambda)_-$ for all $\lambda' \in \mathrm{G}\cdot \lambda$. Therefore, the rank of the vector bundle $W$ is $m=\operatorname{dim}_\mathbb{C}V(\lambda)_-$.
    
    Denote $E = \mathrm{G}\times_{C_\mathrm{G}(\lambda)}V(\lambda)_-$, $p_E: E \to \mathrm{G}/C_\mathrm{G}(\lambda)$, then we have the following commutative diagram
    $$
    \begin{tikzcd}
        E \arrow[r, "\widetilde{\Phi}"] \arrow[d, "p_E" '] & W \arrow[d,"\pi"] \\
        \mathrm{G}/C_\mathrm{G}(\lambda) \arrow[r, "\phi" '] & \mathrm{G}\cdot \lambda
    \end{tikzcd},
    $$
    where $\phi$ is the canonical isomorphism. So $\pi = \phi \circ p_E \circ \widetilde{\Phi}^{-1}$, and thus we can conclude that $\pi$ is a holomorphic G-equivariant vector bundle that is isomorphic to the associated bundle, as stated. 
\end{proof}
Fix $n \geq 0$ and a based map $f_0 : S^n \to V^{st}$ with $f_0(s_0) = x_0$, and let
$$
\mathcal{F} := \{f:S^n \times I \to V \mid f(s,0) = f_0(s), f(s_0,t) = x_0, f(s,1) = x_0 \},
$$
where $s\in S^n$, and $t\in I$, then its tangent space is
$$
T_f\mathcal{F} = \{\dot{f} \in C^r(S^n\times I, V) \mid \dot{f}(s,0)=0, \dot{f}(s_0,t)=0,\dot{f}(s,1)=0 \}.
$$
For $\lambda' \in \mathrm{G}\cdot \lambda$ let $\pi_{\lambda'}^\perp : V \to V(\lambda')_-$ denote the orthogonal projection. Define
$$
D:\mathcal{F} \times S^n \times I \times (\mathrm{G}\cdot\lambda) \to W, \quad D(f,s,t,\lambda') = (\lambda', \pi_{\lambda'}^\perp(f(s,t))).
$$
Let $p=(f,s,t,\lambda')\in 
\mathcal F\times S^{n}\times I\times (\mathrm{G}\cdot\lambda)$, and let $x:=f(s,t)\in V$. Then a tangent vector at $p$ in $T_p(\mathcal{F} \times S^n \times I \times \mathrm{G}\cdot\lambda)$ can be written as $(\dot{f}, v_s, v_t, \xi)$. Then
$$
dD_p(\dot{f}, v_s, v_t, \xi) = (\xi, \pi_{\lambda'}^\perp(\dot{f}(s,t)+v_s\partial_sf(s,t)+v_t\partial_tf(s,t)) + (d\pi_\bullet^\perp)_{(\lambda', x)}(\xi)),
$$
where $(d\pi_\bullet^\perp)_{(\lambda', x)}(\xi) \in V(\lambda')_-$ denotes the derivative, with respect to $\lambda'$, of $\pi_{\lambda'}^\perp$, evaluated at $x\in V$ and at direction $\xi \in T_{\lambda'}(\mathrm{G}\cdot \lambda)$. To deduce the above formula, we start by writing the evaluation map
$$
\operatorname{ev}: \mathcal{F} \times S^n \times I \to V, \quad (f,s,t) \mapsto f(s,t).
$$
Its differential at $(f,s,t)$ is
$$
d(\operatorname{ev}_{(f,s,t)})(\dot{f}, v_s, v_t) = \dot{f}(s,t) + v_s\partial_sf(s,t) + v_t\partial_tf(s,t) \in V.
$$
Next define $\Phi : (\mathrm{G}\cdot \lambda) \times V \to W$ by $\Phi(\lambda',v) := (\lambda', \pi_{\lambda'}^\perp(v))$, and thus by definition $D = \Phi \circ (\operatorname{id}_{\mathrm{G}\cdot\lambda} , \operatorname{ev})$.

Pick an arbitrary $(\lambda', x)$, we have
$$
d\Phi_{(\lambda',x)}(\xi, u) = (\xi, \pi_{\lambda'}^\perp(u) + (d\pi_\bullet^\perp)_{(\lambda',x)}(\xi)),
$$
where $u$ is a tangent vector of $V$. Then at the point $p = (f,s,t,\lambda')$ we have $dD_p = d\Phi_{(\lambda',x)} \circ (\operatorname{id}, d(\operatorname{ev})_{(f,s,t)})$. Substituting $(\dot{f}, v_s, v_t, \xi)$ and use the above equations yields
\begin{align*}\label{eq: 2}
    dD_p(\dot{f}, v_s, v_t, \xi) &= d\Phi_{(\lambda',x)}(\xi, \dot{f}(s,t) + v_s\partial_s f(s,t) + v_t\partial_t f(s,t))\\
    &= (\xi, \pi_{\lambda'}^\perp (\dot f(s,t) + v_{s}\,\partial_{s}f(s,t) + v_{t}\,\partial_{t}f(s,t)) + (d\pi_\bullet^\perp)_{(\lambda', x)}(\xi)), \tag{2}
\end{align*}
as desired.

We now recall some definitions. Let $X$ be a topological space.
\begin{defn}
    A subset $\mathcal{R} \subset X$ is called \emph{residual} if it contains a countable intersection of open dense subsets of $X$. If $X$ is a \emph{Baire space}, for example, a complete metric space \cite{Rudin1991}, then every residual subset of $X$ is dense. We shall say that a property holds \emph{generically} in $X$ if it holds on some residual subset of $X$.
\end{defn}
\begin{prop}\label{Prop 2.4}
    The map $D$ is transverse to the zero section $\mathcal{O}_W \subset W$; i.e. for every $p = (f,s,t,\lambda')$ satisfying $D(p) \in \mathcal{O}_W$ we have
    $$
    dD_p(T_pM) + T_{D(p)}\mathcal{O}_W = T_{D(p)}W,
    $$
    where $M = \mathcal{F} \times S^n \times I \times (\mathrm{G}\cdot\lambda)$. Moreover, $D_f$ is transverse to the zero section $\mathcal{O}_W$ for $f$ in a residual subset of $\mathcal{F}$.
\end{prop}
\begin{proof}
    Choose $p=(f,s,t,\lambda')$ with $D(p) = (\lambda',0) \in \mathcal{O}_W$. Since $W \to \mathrm{G}\cdot\lambda$ is a vector bundle with fibre $V(\lambda')_-$, we have splittings
    $$
    T_{(\lambda',0)}W = T_{\lambda'}(\mathrm{G}\cdot \lambda) \oplus V(\lambda')_-, \quad T_{(\lambda',0)}\mathcal{O}_W = T_{\lambda'}(\mathrm{G}\cdot\lambda) \oplus \{0\}.
    $$
    Recall that the previous computations \eqref{eq: 2} imply that the differential at $p$, where $x=f(s,t)$ and $\pi_{\lambda'}^\perp(x) = 0$, is given by
    $$
    dD_p(\dot{f},v_s, v_t, \xi) = (\xi, \pi_{\lambda'}^\perp (\dot f(s,t) + v_{s}\,\partial_{s}f(s,t) + v_{t}\,\partial_{t}f(s,t)) + (d\pi_\bullet^\perp)_{(\lambda', x)}(\xi)).
    $$

    Now fix an arbitrary vector $v \in V(\lambda')_-$, and choose a bump function $\beta : S^n \times I \to \mathbb{R}$ such that
    \begin{itemize}
        \item $\beta(s, t)=1$,
        \item $\operatorname{supp}\beta$ is contained in a small ball that does not intersect the boundary $(S^n \times \{0,1\}) \cup (\{s_0\} \times I)$.
    \end{itemize}
    Consequently, to show surjectivity onto $V(\lambda')_-$, we consider variations with $\xi = 0$. In this case, $(d\pi_\bullet^\perp)_{(\lambda', x)}(0) = 0$ because this term is linear in $\xi$. Thus, the differential simplifies to
    $$
    dD_p(\dot{f},v_s,v_t,0) = (0, \pi_{\lambda'}^\perp(\dot{f}(s,t) + v_s\partial_sf(s,t) + v_t\partial_tf(s,t))).
    $$
    Now setting $\dot{f} = \beta v$, $v_s = v_t = 0$ gives $\pi_{\lambda'}^\perp (\dot{f}(s,t)) = \pi_{\lambda'}^\perp(v) = v$. Therefore the differential of $D$ at $p$ becomes
    $$
    dD_p(\dot{f},0,0,0) = (0,v).
    $$
    Since $v$ is arbitrary, we have $V(\lambda')_- \subset \operatorname{Im}(dD_p)$. Hence we compute
    $$
    dD_p(T_pM) + T_{(\lambda',0)}\mathcal{O}_W = \operatorname{Im}(dD_p) + (T_{\lambda'}(\mathrm{G}\cdot\lambda) \oplus \{0\}) = T_{\lambda'}(\mathrm{G}\cdot\lambda) \oplus V(\lambda')_- = T_{(\lambda',0)}W,
    $$
    which then implies that $D$ is transverse to $\mathcal{O}_W$ at $p$. And since $p$ was arbitrarily chosen, we conclude that $D$ is transverse to $\mathcal{O}_W$ globally.

Consider $X = S^n \times I \times \mathrm{G}\cdot \lambda$, and define
$$
D_f := D|_{\{f\} \times X} : X \to W, \quad (s, t, \lambda')\mapsto (\lambda', \pi_{\lambda'}^\perp(f(s,t))).
$$
Then the subset $\mathcal{F}^\pitchfork = \{f\in\mathcal{F} \mid D_f \pitchfork \mathcal{O}_W \}$ is residual in $\mathcal{F}$. Therefore $D_f$ is transverse to the zero section $\mathcal{O}_W$ for $f$ in a residual subset by Thom transversality theorem \cite[Theorem 2.1]{Hirsch1976}.
\end{proof}

\begin{cor}\label{cor 2.4}
    Let $f \in \mathcal{F}$ be a generic homotopy. If the rank $m = \operatorname{dim}_\mathbb{C}V(\lambda)_-$ of the vector bundle $W \to \mathrm{G}\cdot \lambda$ satisfies the condition
    \begin{equation}\label{eq: 3}
    n + 1 + 2 \operatorname{dim}_\mathbb{C}(\mathrm{G}\cdot\lambda) < 2m, \tag{3}
    \end{equation}
    then the preimage $D^{-1}_f(\mathcal{O}_W)$ is empty.
\end{cor}
\begin{proof}
    Assume the preimage $D_f^{-1}(\mathcal{O}_W)$ is non-empty. Take any point $x \in D_f^{-1}(\mathcal{O}_W) \subset X = S^n\times I \times (\mathrm{G}\cdot \lambda)$, and set $y = D_f(x) \in \mathcal{O}_W$. Because $D_f$ is transverse to $\mathcal{O}_W$, we have the equality of tangent spaces
    $$
    T_yW = dD_f(T_xX) + T_y\mathcal{O}_W.
    $$
    Hence
    $$
    \dim_\mathbb{R}T_yW = \dim_\mathbb{R}(dD_f(T_xX)) + \dim_\mathbb{R}T_y\mathcal{O}_W - \dim_\mathbb{R}(dD_f(T_xX)\cap T_y\mathcal{O}_W).
    $$
    
    Note that $\dim_\mathbb{R}T_yW = 2m + 2\dim_\mathbb{C}(\mathrm{G}\cdot \lambda)$, $\dim_\mathbb{R}T_y\mathcal{O}_W = 2\dim_\mathbb{C}(\mathrm{G}\cdot \lambda)$, and since the dimension of the intersection $dD_f(T_xX)\cap T_y\mathcal{O}_W$ is non-negative, we obtain the inequality
    $$
    2m \leq \dim_\mathbb{R}(dD_f(T_xX)) \leq \dim_\mathbb{R}T_xX = n + 1 + 2\dim_\mathbb{C}(\mathrm{G}\cdot \lambda),
    $$
    where the second inequality follows from the rank-nullity theorem as the differential $dD_f$ is a linear map.

    The original hypothesis of the corollary is the negation of the inequality obtained above, therefore we have established that $D_f^{-1}(\mathcal{O}_W)$ must be empty.
\end{proof}
$D^{-1}_f(\mathcal{O}_W)$ being empty for the chosen $\mathrm{G} \cdot \lambda$ means that for any point $(s,t)$ in the homotopy and any $\lambda' \in \mathrm{G}\cdot \lambda$, the point $f(s,t)$ is non-zero. Consequently, the Hilbert-Mumford weight $\mu(f(s,t),\lambda') > 0$, which implies that no $\lambda'$ from this particular class can be used to show that $f(s,t)$ is unstable or strictly semistable.

To ensure a homotopy lies entirely within the stable locus, this must hold for every class of 1-PS that can characterise non-stable points. Note that the Hesselink stratification of the non-stable locus $V^{us}(\rho) \cup (V^{ss}(\rho)\setminus V^{st}(\rho))$ is a finite union of strata, each associated with a particular conjugacy class of 1-PSs that can cause instability \cite{HoskinsStratAffine}.

Let $\{[\lambda_1],[\lambda_2],...,[\lambda_k] \}$ be the finite set of conjugacy classes of 1-PSs that characterise all possible ways a point can be $\rho$-unstable or strictly $\rho$-semistable. For each class $[\lambda_j]$, let $W_j$ be the corresponding vector bundle defined in \eqref{eq: 1} whose fibre rank $m_j = \operatorname{dim}_\mathbb{C}V(\lambda_j)_-$ satisfies inequality \eqref{eq: 3}. If the condition
$$
n + 1 + 2\operatorname{dim}_\mathbb{C}(\mathrm{G}\cdot \lambda_j) < 2m_j
$$
holds for each class $[\lambda_j]$ with $j = 1,...,k$, then Corollary \ref{cor 2.4} implies that there is a residual set $\mathcal{F}^\pitchfork_j\subset\mathcal{F}$ of homotopies $f$ such that $D^{-1}_f(\mathcal{O}_{W_j})$ is empty. The intersection $\bigcap^k_{j=1}\mathcal{F}^\pitchfork_j$ is also a residual subset of $\mathcal{F}$ because a countable intersection of residual sets is residual. Any homotopy $f$ in the intersection avoids every Hesselink stratum, which implies that its image lies entirely in the stable locus $V^{st}(\rho)$.

\begin{thm}\label{Thm 2.6}
    Define $d_{\min} = \min_{\substack{j \in \{1,\dots,k\}}} \{2m_j - 2\operatorname{dim}_\mathbb{C}(\mathrm{G}\cdot \lambda_j)\}$. Let $n$ be a non-negative integer such that $n+1 < d_{\min}$, then the $n$-th homotopy group of the stable locus is trivial, in other words,
    $$
    \pi_n(V^{st}(\rho)) = 0.
    $$
    Equivalently, $V^{st}(\rho)$ is $(d_{\min}-2)$-connected.
\end{thm}

This theorem allows us to deduce the homotopy groups of the moduli space of stable orbits, $\mathcal{M}^{st}(\rho) := V^{st}(\rho) / \mathrm{G}$. Note that the principal $\mathrm{G}$-bundle $\mathrm{G}\to V^{st}(\rho) \to \mathcal{M}^{st}(\rho)$ induces a long exact sequence in homotopy groups
$$
\cdots \to \pi_n(\mathrm{G}) \to \pi_n(V^{st}(\rho)) \to \pi_n(\mathcal{M}^{st}(\rho)) \to \pi_{n-1}(\mathrm{G}) \to \pi_{n-1}(V^{st}(\rho)) \to \cdots,
$$
we have also assumed that $\pi_k(V^{st}(\rho)) = 0$ for $k \leq d_{\min} - 2$. Then if $k = 0$, which means $0 \leq d_{\min} - 2$, then $\pi_0(V^{st}(\rho)) = 0$. Thus we have $0 \to \pi_0(\mathcal{M}^{st}(\rho)) \to 0$, which then implies that $\pi_0(\mathcal{M}^{st}(\rho)) = 0$. Hence $\mathcal{M}^{st}(\rho)$ is path-connected.

For higher homotopy groups, consider $1\leq k < d_{\min} - 2$. We observe from the long exact sequence that
\begin{itemize}
    \item if $k \leq d_{\min} - 2$, then $\pi_k(V^{st}(\rho)) = 0$;
    \item if $k-1 \leq d_{\min}-3 < d_{\min}-2$, then $\pi_{k-1}(V^{st}(\rho)) = 0$. The corresponding segment of the long exact sequence $\pi_k(V^{st}(\rho)) \to \pi_k(\mathcal{M}^{st}(\rho)) \to \pi_{k-1}(\mathrm{G}) \to \pi_{k-1}(V^{st}(\rho))$ becomes $0 \to \pi_k(\mathcal{M}^{st}(\rho)) \to \pi_{k-1}(\mathrm{G}) \to 0$. The exactness implies that the connecting homomorphism $\partial_k : \pi_k((\mathcal{M}^{st}(\rho)) \to \pi_{k-1}(\mathrm{G})$ is an isomorphism.
\end{itemize}
This leads to the following corollary:
\begin{cor}\label{Cor 2.8}
    Assume $\mathrm{G}$ acts freely on $V^{st}(\rho)$ with quotient $\mathcal{M}^{st}(\rho) := V^{st}(\rho)/\mathrm{G}$, and let $d_{\min}$ be defined as in the preceding theorem. If $n$ is a non-negative integer such that $n+1 < d_{\min}$, then the homotopy groups of the quotient $\mathcal{M}^{st}(\rho)$ are related to the homotopy groups of $\mathrm{G}$ by an isomorphism
    $$
    \pi_n(\mathcal{M}^{st}(\rho)) \cong \pi_{n-1}(\mathrm{G})
    $$
    for $1 \leq n < d_{\min}-1$. Furthermore, if $d_{\min} > 1$, then $\mathcal{M}^{st}(\rho)$ is path-connected.
\end{cor}

\section{Applications}
\subsection{Stability and Controllability of Linear Control Systems}
In this example, we prove the following result.

\begin{lemma}
    Let $m, n \geq 1$ and assume $m\geq n$. Let
    $$
    V^c = \{(A, B)\in M_{n \times n}(\mathbb{C}) \times M_{n \times m}(\mathbb{C}) \mid (A, B)\text{ is controllable } \}.
    $$
    Then the controllable locus $V^c$ is $(2m - 2n)$-connected.
\end{lemma}

Fix integers $n, m \geq 1$, where $n$ denotes the dimension of the state space, and $m$ denotes the dimension of the input space. A linear control system with state space $\mathbb{C}^n$ and input space $\mathbb{C}^m$ is determined by a pair of matrices
$$
(A, B) \in M_{n \times n}(\mathbb{C})  \times M_{n \times m}(\mathbb{C}).
$$
We consider the affine space
$$
V := M_{n \times n}(\mathbb{C})  \times M_{n \times m}(\mathbb{C}),
$$
and the group $\mathrm{G} = \operatorname{GL}_n(\mathbb{C})$ acts on $V$ by change of basis in the state space, i.e. for any $(A, B) \in V$, we have $g\cdot (A, B) = (gAg^{-1} , gB)$ for all $g\in \mathrm{G}$.

The pair $(A, B)$ is said to be \emph{controllable} if and only if the \emph{controllability matrix}
$$
C(A, B) = [B, AB, A^2B, ...,A^{n-1}B] \in M_{n \times nm}(\mathbb{C})
$$
has rank $n$ \cite{Geiss2006ModuliQuivers}. Controllability is the fundamental property of a control system that ensures it can be driven from any initial state to any desired final state within a finite time.

Choose the character $\chi : \operatorname{GL}_n(\mathbb{C}) \to \mathbb{C}^\times$ defined by $\chi(g) = \operatorname{det}(g)$, and $V$ can be endowed with the $\chi$-linearisation, in other words, the $\mathrm{G}$-action can be lifted from $V$ to $V\times \mathbb{C}$ by defining $g \cdot (v, \zeta) := (g \cdot v, \chi(g)^{-1}\zeta)$ for any $(v, \zeta)\in V\times \mathbb{C}$. Then a point $v \in V$ is \emph{$\chi$-semistable} if
$$
\overline{\mathrm{G}\cdot(v, 1)} \cap V\times \{0\} = \emptyset.
$$
Further, a point $v\in V$ is \emph{$\chi$-stable} if it is $\chi$-semistable, its orbit $\mathrm{G}\cdot v$ is closed within $V^{ss} \subset V$, and its stabiliser subgroup $\mathrm{G}_v = \{g\in \mathrm{G} \mid g\cdot v = v\}$ is finite.

Proposition 4.1 in \cite{Geiss2006ModuliQuivers} gives the criterion that links the GIT stability with the controllability of linear control systems. A pair $(A, B)$ is stable in the sense of GIT if and only if it is controllable, which is equivalent to the condition that no proper $A$-invariant subspace contains $\operatorname{Im}B \subset \mathbb{C}^n$. This can be seen by considering the contrapositive. If $(A, B)$ is not controllable, there exists a proper $A$-invariant subspace $W\subset \mathbb{C}^n$ such that $\operatorname{Im}B \subseteq W$. If $r = \operatorname{dim}W$ with $1\leq r \leq n-1$, then we can choose a basis of $\mathbb{C}^n$ so that $W = \operatorname{Span}\{ e_1,...,e_r\}$, $W' = \operatorname{Span}\{e_{r+1},...,e_n\}$. In this adapted basis, the matrices $A$ and $B$ take the block form
$$
A = \begin{pmatrix}
    A_{11} & A_{12} \\ 0 & A_{22}
\end{pmatrix},
\quad
B=\begin{pmatrix}
    B_1 \\ 0
\end{pmatrix}
$$
where $A_{11} \in M_{r \times r}(\mathbb{C})$, $A_{12} \in M_{r \times (n-r)}(\mathbb{C})$, $B_1 \in M_{r \times m}(\mathbb{C})$, and the zero block in $B$ comes from $\operatorname{Im}B \subseteq W$.

Consider the one-parameter subgroup $\lambda:\mathbb{C}^\times \to \operatorname{GL}_n(\mathbb{C})$ with $\lambda(t) = \operatorname{diag}(1,...,1,\zeta^{-1},...,\zeta^{-1})$, where 1 appears $r$ times and $\zeta^{-1}$ appears $n-r$ times. Clearly $\lambda$ preserves the block decomposition. Then under conjugation, we have $\lambda(\zeta) \cdot (A, B) = (\lambda(\zeta)A\lambda(\zeta)^{-1}, \lambda(\zeta)B)$, where in blocks we yield
$$
\lambda(\zeta)A\lambda(\zeta)^{-1} = \begin{pmatrix}
    A_{11} & \zeta A_{12} \\ 0 & A_{22}
\end{pmatrix},
\quad
\lambda(\zeta)B = \begin{pmatrix}
    B_1 \\ 0
\end{pmatrix}.
$$
Then the limit $\operatorname{lim}_{\zeta\to 0}\lambda(\zeta)\cdot (A, B)$ exists and the numerical invariant
$$
\langle \chi, \lambda \rangle = -(n-r) < 0
$$
since $r < n$. Thus, by the Hilbert-Mumford criterion, $(A, B)$ is not $\chi$-semistable and thus not $\chi$-stable. Therefore, we can define the unstable locus associated to the subspace W as
$$
S_W := \{(A, B) \in V \mid A(W)\subset W, \operatorname{Im}B \subset W \}.
$$

Conversely, assume there is no proper linear subspace $W \subset \mathbb{C}^n$ such that $A(W) \subseteq W$ and $\operatorname{Im}B \subseteq W$. We show that $(A, B)$ is $\chi$-stable. By the Hilbert-Mumford criterion, it suffices to check the relevant numerical inequality for every one-parameter subgroup $\lambda: \mathbb{C}^\times \to \operatorname{GL}_n(\mathbb{C})$ for which the limit
$$
\lim_{\zeta \to 0}\lambda(\zeta) \cdot (A, B)
$$
exists in $V$, and then to verify that the stabiliser of $(A, B)$ is finite.

Let $\lambda$ be such a one-parameter subgroup. Since every one-parameter subgroup of $\operatorname{GL}_n(\mathbb{C})$ is diagonalisable, we may choose a basis $\{e_1,...,e_n\}$ of $\mathbb{C}^n$ such that $\lambda(\zeta)$ is diagonal with respect to this basis, i.e., $\lambda(\zeta)e_i = \zeta^{w_i}e_i$ for some integers $w_i$. Consider the matrix entries of $A$ and $B$ with respect to this basis. Under the action
$$
\lambda(\zeta) \cdot (A, B) = (\lambda(\zeta)A\lambda(\zeta)^{-1}, \lambda(\zeta)B),
$$
the entry $A_{ij}$, which corresponds to the component $e_j \mapsto e_i$, has weight $w_i - w_j$, while the $i$-th row of $B$ has weight $w_i$.

The existence of limit as $\zeta \to 0$ then implies that $A_{ij} = 0$ when $w_i - w_j < 0$ and that the $i$-th row of $B$ is zero when $w_i < 0$. Define
$$
V_{\geq 0} = \operatorname{Span}\{e_i \mid w_i\geq 0\} \subseteq \mathbb{C}^n.
$$
The condition on the rows of $B$ implies $\operatorname{Im}B \subseteq V_{\geq 0}$. Moreover, if $e_j \in V_{\geq 0}$, then $w_j \geq 0$. Any component of $Ae_j$ in a basis vector $e_i$ with $w_i < 0$ would have weight $w_i - w_j < 0$, and hence vanishes. Therefore, $A(V_{\geq 0}) \subseteq V_{\geq 0}$. Thus $V_{\geq 0}$ is an $A$-invariant subspace that contains $\operatorname{Im}B$. If $V_{\geq 0}$ were proper, it would contradict the assumption. Hence $V_{\geq 0} = \mathbb{C}^n$, and every weight satisfies $w_i \geq 0$.

For the character $\chi(g) = \det (g)$, we have $\chi(\lambda(\zeta)) = \det(\lambda(\zeta)) = \zeta^{w_1 + \cdots + w_n}$. Thus
$$
\langle \chi, \lambda\rangle = \sum_{i = 1}^n w_i \geq 0.
$$
Furthermore, if $\langle \chi, \lambda\rangle = 0$, then since all $w_i \geq 0$, we have $w_1 = \cdots = w_n = 0$. Hence $\lambda$ is the trivial 1-PS, and so every non-trivial 1-PS $\lambda$ for which the limit exists satisfies $\langle \chi, \lambda\rangle > 0$.

It remains to check that the stabiliser is finite. Suppose an element $g \in \mathrm{G}$ lies in the stabiliser, then $gAg^{-1} = A$ and $gB = B$. Note that the first equality is equivalent to $gA = Ag$, hence for every $j\geq 0$, we have $gA^jB = A^jgB = A^jB$, which means $g$ fixes every column of the controllability matrix $[B, AB, \ldots, A^{n-1}B]$. Since $(A, B)$ is controllable, this matrix has rank $n$, so its columns span $\mathbb{C}^n$. It follows that $g = I_n$, hence the stabiliser of $(A, B)$ is trivial. In particular, it is finite. Therefore, we conclude that $(A, B)$ is $\chi$-stable.

We have now proved that
$$
(A, B) \in V^{st}(\chi) \Longleftrightarrow (A, B) \text{ is controllable},
$$
and hence $V^{st}(\chi) = V^{c} = \{(A, B) \in V \mid (A, B)\text{ is controllable} \}$. Consequently, the unstable locus $V^{us}(\chi) = V \setminus V^{st}(\chi)$ is the locus of uncontrollable systems.

By the preceding discussion, a pair $(A, B) \in V$ is unstable if and only if there exists a proper subspace $W \subset \mathbb{C}^n$ such that $A(W) \subseteq W$ and $\operatorname{Im}B \subseteq W$. It remains to apply the transversality argument to the unstable locus. For $0 \leq r \leq n-1$, let $\lambda_r: \mathbb{C}^\times \to \operatorname{GL}_n(\mathbb{C})$ be the one-parameter subgroup
$$
\lambda_r(\zeta) = \operatorname{diag}(I_r, \zeta^{-1}I_{n-r}).
$$
Fix a decomposition $\mathbb{C}^n = W \oplus W'$, where $\dim W = r$ and $\dim W' = n - r$. For an arbitrary pair $(A, B) \in V$, let
$$
A = \begin{pmatrix}
    A_{11} & A_{12} \\ A_{21} & A_{22}
\end{pmatrix},
\quad
B=\begin{pmatrix}
    B_1 \\ B_2
\end{pmatrix},
$$
where $A_{21} : W \to W'$ and $B_2 : \mathbb{C}^m \to W'$. Unlike the previous situation, here we do not assume $A_{21} = 0$ or $B_2 = 0$. Observe that under the action of $\lambda_r(\zeta) = \operatorname{diag}(I_r, \zeta^{-1}I_{n-r})$, the components of $A_{21}$ and $B_2$ have negative weight. Hence
$$
V(\lambda_r)_- \cong \operatorname{Hom}(W, W')\oplus \operatorname{Hom}(\mathbb{C}^m, W').
$$
Therefore the complex dimension of the negative weight space is
$$
\dim_\mathbb{C}V(\lambda_r)_- = \dim_\mathbb{C}A_{21} + \dim_\mathbb{C}B_2 = r(n-r) + m(n-r) = (n-r)(r+m).
$$
On the other hand, the conjugation orbit of $\lambda_r$ is $\mathrm{G}\cdot \lambda_r \cong \mathrm{G}/C_\mathrm{G}(\lambda_r)$. For $0 < r < n$, one has
$$
C_\mathrm{G}(\lambda_r) \cong \operatorname{GL}_r(\mathbb{C}) \times \operatorname{GL}_{n-r}(\mathbb{C}).
$$
Hence
$$
\dim_\mathbb{C}(\mathrm{G}\cdot \lambda_r) = n^2 - r^2 - (n - r)^2 = 2r(n - r).
$$
In particular, if $r = 0$, then we would yield $\dim_\mathbb{C}(\mathrm{G}\cdot \lambda_r) = 0$.

Let $\mathcal{W}_r \to \mathrm{G}\cdot \lambda_r$ be the $\mathrm{G}$-equivariant holomorphic vector bundle whose fibre over $\mu\in \mathrm{G}\cdot \lambda_r$ is the negative weight space $V(\mu)_-$. Since all $\mu \in \mathrm{G}\cdot \lambda_r$ are conjugate to $\lambda_r$, the rank of this bundle is
$$
\operatorname{rank}_\mathbb{C}\mathcal{W}_r = \dim_\mathbb{C}V(\lambda_r)_- = (n - r)(r + m).
$$

Fix $k\geq 0$, and let $\mathcal{F}$ denote the space of homotopies $F: S^k \times I \to V$ with the aforementioned boundary conditions. As in Section 2, define the evaluation map $D_r : \mathcal{F} \times S^k \times I \times (\mathrm{G}\cdot \lambda_r) \to \mathcal{W}_r$ by
$$
D_r(F, s, t, \mu) = (\mu, \pi_\mu^\perp(F(s, t))),
$$
where the map $\pi_\mu^\perp: V \to V(\mu)_-$ denotes the projection onto the negative weight space for $\mu$. By Proposition \ref{Prop 2.4}, the map $D_r$ is transverse to the zero section $\mathcal{O}_{\mathcal{W}_r} \subset \mathcal{W}_r$. Therefore, for a generic homotopy $F$, the restriction $D_r|_F : S^k \times I \times (\mathrm{G}\cdot \lambda_r) \to \mathcal{W}_r$ is transverse to the zero section.

The condition $D_r(F,s,t,\mu) \in \mathcal{O}_{\mathcal{W}_r}$ is equivalent to $\pi_\mu^\perp(F(s,t)) = 0$, which means $F(s,t)\in \bigoplus_{i\geq 0}V_i(\mu)$. This is the condition for the limit $\lim_{\zeta\to 0}\mu(\zeta) \cdot F(s,t)$ to exist in $V$.
By Corollary \ref{cor 2.4}, the preimage $D_r|_F^{-1}(\mathcal{O_W})$ is empty when
$$
\operatorname{dim}_\mathbb{R}S^k + 1 + 2(\operatorname{dim}_\mathbb{C}(\mathrm{G}\cdot\lambda_r)) < 2\dim_\mathbb{C}V(\lambda_r)_-.
$$
Substituting the above dimensions gives
$$
k + 1 + 2\cdot 2r(n-r) < 2(n - r)(r + m).
$$
Hence we obtain $k + 1 < 2(n - r)(m - r)$. Recall that the unstable locus is obtained by allowing proper subspaces $W \subset \mathbb{C}^n$ which contain $\operatorname{Im} B$ and are preserved by $A$. In other words, it is determined by the possible dimensions $0 \leq r \leq n-1$ of such destabilising subspaces. Therefore, the minimum is
$$
d_{\min} = \min_{0 \leq r \leq n-1} 2(n-r)(m - r).
$$

Assume now that $m \geq n$. Then let $q = n - r$, we have $2(n-r)(m-r) = 2q(m-n+q)$ for $1 \leq q \leq n$. Since $m-n$ is non-negative, this expression achieves its minimum when $q = 1$, which is equivalent to when $r = n - 1$. Therefore,
$$
d_{\min} = 2(m - n + 1) = 2m - 2(n - 1).
$$
It follows from Corollary \ref{cor 2.4} that if $k + 1 < d_{\min} = 2m - 2(n - 1)$, then a generic homotopy $F$ can be chosen so that its image avoids the entire unstable locus $V^{us}$. Since $V$ is an affine space, it is contractible. Hence every map $S^k \to V^{st}(\chi)$ has a null-homotopy in $V$. Therefore
$$
\pi_k(V^{st}(\chi)) = 0 \text{ for all }k+1 < 2m - 2(n - 1).
$$
Since $V^{st}(\chi) = V^c$, we conclude that, for $m \geq n$, the controllable locus $V^c$ is $(2m - 2n)$-connected.

\subsection{Directed Gaussian Graphical Models and MLE}
We prove the following result in this example.
\begin{lemma}
    The moduli space of stable samples $\mathcal{M}^{st} = V^{st}/\mathrm{G}$ admits non-trivial families of stabilisations that are non-contractible, i.e. topologically distinct.
\end{lemma}
The work of Améndola et al. \cite{AmendolaKohnReichenbachSeigal2021} provides a novel connection between Maximum Likelihood Estimation (MLE) in statistics and the principles of Geometric Invariant Theory (GIT). Their idea is to translate statistical properties, such as the existence and uniqueness of the MLE, into the language of GIT stability by using the Kempf-Ness functional. They apply this to Gaussian group models, where maximising the log-likelihood is shown to be equivalent to a norm minimisation problem over a group orbit \cite[Proposition 3.4, Proposition 3.13]{AmendolaKohnReichenbachSeigal2021}. This relates the boundedness of the likelihood to semistability, the existence of an MLE to polystability, and the uniqueness of the MLE to stability \cite[Theorem 3.10, Theorem 3.15]{AmendolaKohnReichenbachSeigal2021}.

Building upon this framework, Derksen and Makam \cite{DerksenMakam2021} employed the machinery of quiver representations to solve the MLE threshold problem for matrix normal models, which is a key example of a Gaussian group model. They provided exact formulae for the number of samples required to ensure an MLE exists and is unique \cite[Theorem 1.2, Theorem 1.3]{DerksenMakam2021}, and thus proving a conjecture of Drton, Kuriki, and Hoff \cite[Corollary 1.4]{DerksenMakam2021}. Related work of Makam, Reichenbach, and Seigal \cite{makam2022symmetriesdirectedgaussiangraphical} studies Gaussian graphical models on directed acyclic graphs with symmetry constraints, namely RDAG (restricted directed acyclic graph) models, and characterises existence and uniqueness of the MLE via linear independence conditions. The following example applies this GIT perspective to a specific Directed Acyclic Graph (DAG) model, demonstrating the topological properties of the GIT-stable space of data samples for which a unique MLE exists.

We consider a directed Gaussian graphical model on a star-shaped Directed Acyclic Graph (DAG), which is a connected DAG that has a unique child vertex, which is a vertex that has a parent, meaning there is an edge pointing to it from another vertex. This graph has $k$ parent vertices $\{1,...,k \}$ and a single child vertex $k+1$, representing a linear regression problem \cite{bérczi2023completecollineationsmaximumlikelihood}.

Suppose that sample data are collected into a matrix $Y = [Y^{(1)}\cdots Y^{(k)} \mid Y^{(k+1)}]$ of size $n \times (k+1)$, where $n$ denotes the number of observations. The Maximum Likelihood Estimation (MLE) for the regression coefficients exists and unique if and only if the matrix formed by the parent columns, $[Y^{(1)} \cdots Y^{(k)}]$, has full column rank \cite[Theorem 3.5 (c)]{bérczi2023completecollineationsmaximumlikelihood}. We consider a degenerate sample $f = [Y^{(1)}\cdots Y^{(k)} \mid Y^{(k+1)}]$, where the parent block $X := [Y^{(1)} \cdots Y^{(k)}]$ is rank-deficient. Thus there are infinitely many regression coefficients $\beta$ satisfying $X^\top(Y^{(k+1)}-X\beta) = 0$ \cite[Section 3.2 \& Section 7]{bérczi2023completecollineationsmaximumlikelihood}. We refer to Section 5 of \cite{bérczi2023completecollineationsmaximumlikelihood} for the complete construction of the $f$-stabilisation, where Lemma 5.3 shows $\tilde{f}(\epsilon)$ has full column rank, hence lies in $V^{st}$ and has a unique MLE.

Let $V = M_{n\times (k+1)}(\mathbb{C})$ be the sample space, in which a point is a block $[Y^{(1)}\cdots Y^{(k)} \mid Y^{(k+1)}]$ with each column $Y^{(i)} \in \mathbb{C}^n$. Let the reductive group be $\mathrm{G} := \operatorname{GL}_k(\mathbb{C}) \times \mathbb{C}^\times$, which is embedded in $\operatorname{GL}_{k+1}(\mathbb{C})$ via
$$
g = (A, t) \mapsto \begin{pmatrix}
    A & 0 \\
    0 & t^{-1}
\end{pmatrix}.
$$
We define the right action of $\mathrm{G}$ on the sample space $V$ by matrix multiplication as $Y \mapsto Yg$. The character $\rho : \mathrm{G} \to \mathbb{C}^\times$ is then defined by $\rho(g) = t$. Note that the 1-PS $\lambda$ in $\mathrm{G}$ takes the general form $$
\lambda(t) = \operatorname{diag}(t^{w_1}, t^{w_2}, \ldots, t^{w_k}, t^{-\sum_{i=1}^k w_i}),
$$
where the integers $\{w_1, \ldots, w_k, -\sum_{i=1}^k w_i\}$ are the weights of the action.

There exist three equivalences between the notions of GIT stability and the existence and uniqueness of the MLE in the setting of DAG models \cite[Remark 3.6]{bérczi2023completecollineationsmaximumlikelihood}, which can be stated as follows
\begin{itemize}
    \item A sample $Y$ is unstable if no MLE exists;
    \item a sample $Y$ is polystable if an MLE exists, but possible not unique;
    \item a sample $Y$ is stable if the MLE is unique.
\end{itemize}

A 1-PS is destabilising if the limit $\lim_{t\to 0}Y\cdot \lambda(t)$ exists for a non-zero $Y$. The minimal destabilisation corresponds to making a single parent column redundant, namely the corresponding data of this parent column is zero for the limit to exist. We consider the specific 1-PS $\lambda_1$ defined as $\lambda_1(t) := \operatorname{diag}(t^{-1}, 1,\ldots, t)$, where the only negative weight is $-1$, this means that $\lambda_1(t)$ acts on the first parent column $Y^{(1)}$. The corresponding negative weight space is
$$
V(\lambda_1)_- = \{[Y^{(1)}0\cdots 0] \mid Y^{(1)} \in \mathbb{C}^n \}\cong \mathbb{C}^n,
$$
as each sample $Y$ has $n$ observations. Therefore, the dimension of the negative weight space is $\dim_\mathbb{C}V(\lambda_1)_- = n$.

Observe that the 1-PS $\lambda_1$ gives the weight decomposition as
$$
V_- = \operatorname{span}\{e_1\} , \quad V_0 = \operatorname{span}\{e_2,\ldots ,e_{k} \}, \quad V_+ = \operatorname{span}\{ e_{k+1} \},
$$
and any matrix commuting with $\lambda_1$ preserves each weight space. Let $g = (g_{ij})_{1\leq i, j\leq k+1}$ be an arbitrary element in $\mathrm{G}$, then the commutator relation $t^{w_i}g_{ij} = t^{w_j}g_{ij}$ for all $t\in \mathbb{C}^\times$ implies that the matrix $A$ is of the form
$$
A = \begin{pmatrix}
    a & 0\\
    0 & B
\end{pmatrix},
$$
where $a\in\mathbb{C}^\times$ and $B \in \operatorname{GL}_{k-1}(\mathbb{C})$. Hence the centraliser $C_\mathrm{G}(\lambda_1)$ is isomorphic to the product $\mathbb{C}^\times \times \operatorname{GL}_{k-1}(\mathbb{C}) \times \mathbb{C}^\times$. Thus we have $\dim_\mathbb{C}C_\mathrm{G}(\lambda_1) = 1 + (k-1)^2 + 1 = (k-1)^2 + 2$. Furthermore, since $\mathrm{G} = \operatorname{GL}_k(\mathbb{C}) \times \mathbb{C}^\times$, its dimension is the sum of the dimensions of the two Lie groups, namely $\dim_\mathbb{C}G=k^2+1$. Note that the quotient $\mathrm{G}/C_\mathrm{G}(\lambda_1) \cong \mathrm{G}\cdot \lambda_1$ is irreducible, it follows that $\dim_\mathbb{C}(\mathrm{G}\cdot \lambda_1) = \dim_\mathbb{C}G - \dim_\mathbb{C}C_\mathrm{G}(\lambda_1) = 2k - 2$.

Recall that Proposition \ref{Prop 2.4} establishes that $D_f: S^{\ell} \times I \times \mathrm{G}\cdot\lambda_1$ is transverse to the zero section $\mathcal{O}_W$ for generic homotopies $f$, where $W \cong (\mathrm{G}\cdot \lambda_1) \times \mathbb{C}^n$. Moreover, if $\ell + 1 + 2\dim_\mathbb{C}(\mathrm{G}\cdot\lambda_1) = \ell + 4k - 3 < 2n$, then Corollary 2.4 implies that $D_f^{-1}(\mathcal{O}_W)$ is empty, and so f can be deformed such that its image remains in the stable locus $V^{st}$.

By Theorem \ref{Thm 2.6}, we have $d_{\min} = 2\dim_\mathbb{C}V(\lambda_1)_- - 2\dim_\mathbb{C}(\mathrm{G} \cdot \lambda_1) = 2n-4k+4$. Thus to ensure $d_{\min} > 0$, one requires that $n > 2k - 2$. Consequently, we have $\pi_q(V^{st}) = 0$ for every $0\leq q <d_{\min} - 2$, which implies that $V^{st}$ is $(2n - 4k +2)$-connected. In particular, the stable locus $V^{st}$ is path-connected if $n \geq 2k - 1$, and it is simply connected if $n \geq 2k$. These results imply that the space of samples with a unique MLE is connected, which means that any sample with a unique MLE can be continuously deformed into any other such sample through a path consisting entirely of samples with unique MLEs. Furthermore, the algebraic parameter space $X_f$ of all $f$-stabilisations \cite[Definition 5.8]{bérczi2023completecollineationsmaximumlikelihood} lies in a connected ambient space and inherits the connectedness properties.

Consequently, observe from Corollary \ref{Cor 2.8}, we obtain $\pi_q(V^{st}/\mathrm{G}) \cong \pi_{q-1}(\mathrm{G})$ for all $1\leq q < d_{\min}-1$. Note also that $\mathrm{G} = \operatorname{GL}_k(\mathbb{C}) \times \mathbb{C}^\times$ is homotopy equivalent to $\operatorname{U}(k) \times S^1$, we can therefore conclude that for $1\leq q <d_{\min}-1$,
$$
\pi_q(V^{st}/\mathrm{G}) \cong \begin{cases}
    \mathbb{Z}\oplus \mathbb{Z} \quad q=2,\\
    \mathbb{Z} \quad q\text{ is even, } 4\leq q < 2k,\\
    0 \quad \text{otherwise}.
\end{cases}
$$
Therefore we have obtained the homotopy groups of the moduli space of stable samples.

In particular, $\pi_2(V^{st}/\mathrm{G}) \cong \mathbb{Z} \oplus \mathbb{Z}$ implies that there can be non-trivial families of $f$-stabilisations parametrised by a $2$-sphere, $S^2$, that cannot be continuously contracted to a single model within the moduli space. Moreover, the non-trivial homotopy group $\pi_q \cong \mathbb{Z}$ implies that for each even $q$ with $4\leq q < 2k$, one can construct a family of $f$-stabilisations parametrised by a $q$-sphere, $S^q$, in a non-trivial way, i.e. they are not null-homotopic. Therefore, while the $f$-stabilisation method of Bérczi et al. \cite{bérczi2023completecollineationsmaximumlikelihood} guarantees a solution with a unique MLE, the topological properties of the solution space are intricate. The existence of these non-trivial homotopy groups proves that there are entire families of such stabilisations, parametrised by spheres, which are topologically distinct and cannot be continuously deformed into one another.

From a statistical point of view, the stable locus $V^{st}$ is the space of samples for which the model has a unique MLE. Therefore, Theorem \ref{Thm 2.6} yields the following consequence.
\begin{thm}
    If the sample size is sufficiently large compared to the number of parameters so that the bound of Theorem \ref{Thm 2.6} applies, the locus of samples admitting a unique MLE is $(d_{\min} - 2)$-connected.
\end{thm}
Moreover, this example is also related to the stabilisation procedure of Derksen and Makam's quiver-based statistics work \cite{DerksenMakam2021} and to the $f$-stabilisations of Bérczi et al. \cite{bérczi2023completecollineationsmaximumlikelihood}, which produce perturbations of non-identifiable samples into the stable locus, where the MLE is unique.

\section{Applications to Quivers and Helmke Systems}\label{Helmke systems Section}
\subsection{Quivers and Helmke Systems} In this section, we apply our framework to the setting of Helmke systems discussed in \cite{Bader2008}. Helmke \cite{Helmke1993} generalises the classical equations of a linear dynamical system by introducing additional matrices $E\in M_{n \times n}(\mathbb{C})$ and $F \in M_{p \times p}(\mathbb{C})$, while keeping the classical system through the inclusion
$$
(A, B, C, D) \mapsto (I_n, A, B, C, D, I_p).
$$
After restricting to controllable systems and passing to the corresponding GIT quotient, the moduli space of controllable classical linear dynamical systems embeds as an open subset of a smooth projective compactification \cite{Bader2008}.

\begin{defn}\cite[Definition 3.1]{Bader2008}\label{Controllability Definition}
    A \emph{Helmke system} of type $(n, m, p)$ is a 6-tuple $H = (E, A, B, C, D, F)$, consisting of matrices $E, A \in M_{n \times n}(\mathbb{C})$, $B \in M_{n \times m}(\mathbb{C})$, $C \in M_{p \times n}(\mathbb{C})$, $D \in M_{p \times m}(\mathbb{C})$, $F \in M_{p \times p}(\mathbb{C})$.
\end{defn}

Let $\mathcal{H}_{n, m, p}$ denote the space of all Helmke systems, and $\mathrm{GL}_{n, n, p}$ denotes the product $\mathrm{GL}_n(\mathbb{C}) \times \mathrm{GL}_n(\mathbb{C}) \times \mathrm{GL}_p(\mathbb{C})$. The action of $\mathrm{GL}_{n, n, p}$ on $\mathcal{H}_{n, m, p}$ is defined as
$$
(g_0, g_1, g_2) \cdot (E, A, B, C, D, F) = (g_1Eg_0^{-1}, g_1Ag_0^{-1}, g_1B, g_2Cg_0^{-1}, g_2D, g_2F).
$$

\begin{defn}\cite[Definition 3.2]{Bader2008}
    A Helmke system $H = (E, A, B, C, D, F) \in \mathcal{H}_{n, m, p}$ is called \emph{controllable} if the following conditions are satisfied:
    \begin{enumerate}
        \item $\det(\alpha E + \beta A) \neq 0$ for some $(\alpha, \beta)\in \mathbb{C}^2$,
        \item $\operatorname{rank}\begin{pmatrix} \alpha E+\beta A & B \end{pmatrix}= n$ for all $(\alpha, \beta) \in \mathbb{C}^2\setminus \{0\}$,
        \item $\operatorname{rank}\begin{pmatrix} F & C & D \end{pmatrix} = p$.
    \end{enumerate}
\end{defn}

Consider the Helmke system of type $(2, 3, 1)$. Then we have the space of Helmke systems
$$
\mathcal{H}_{2, 3, 1} = M_{2\times 2}(\mathbb{C})^2 \oplus M_{2 \times 3}(\mathbb{C}) \oplus M_{1 \times 2}(\mathbb{C}) \oplus M_{1 \times 3}(\mathbb{C}) \oplus M_{1\times 1}(\mathbb{C}),
$$
which gives $\dim_\mathbb{C}\mathcal{H}_{2, 3, 1} = 20$. Moreover, we have the following diagram \cite[Remark 3.6]{Bader2008}.

\[\begin{tikzcd}
	{\bullet v_0} &&& {\bullet v_2} &&& {\circ\, v_4} \\
	\\
	\\
	{\bullet v_1} &&& {\circ\, v_3}
	\arrow["C", from=1-1, to=1-4]
	\arrow["E"', bend right=20, from=1-1, to=4-1]
	\arrow["A", bend left=20, from=1-1, to=4-1]
	\arrow["F"', from=1-7, to=1-4]
	\arrow["D"', from=4-4, to=1-4]
	\arrow["B", from=4-4, to=4-1],
\end{tikzcd}\]
where $v_0, v_1 \cong \mathbb{C}^2$, $v_2, v_4 \cong \mathbb{C}$, and $v_3 \cong \mathbb{C}^3$. Note that the reductive group is $\mathrm{GL}_{2,2,1}$, and its action on $\mathcal{H}_{2, 3, 1}$ is
$$
(g_0, g_1, g_2) \cdot (E, A, B, C, D, F) = (g_1Eg_0^{-1}, g_1Ag_0^{-1}, g_1B, g_2Cg_0^{-1}, g_2D, g_2F).
$$

Bader proves a general result of controllable Helmke systems of type $(n, m, p)$ \cite[Proposition 3.7]{Bader2008}. For Helmke systems of type $(n, m, p)$, we have the following:
\begin{prop}\label{Propo 3.7 Bader}
    Let $V = \mathcal{H}_{2, 3, 1}$ and $\mathrm{G} = \operatorname{GL}_2(\mathbb{C}) \times \operatorname{GL}_2(\mathbb{C}) \times \operatorname{GL}_1(\mathbb{C})$. Let $\chi = (r, s, t) \in \mathbb{Z}^3$ be a character of $\mathrm{G}$, written as
    $$
    \chi(g_0, g_1, g_2) = \det(g_0)^r\det(g_1)^s\det(g_2)^t.
    $$
    If $2r + s + t < 0$, $r + s > 0$, and $t > 0$, then for any Helmke system $H \in \mathcal{H}_{2, 3, 1}$, the following are equivalent:
    \begin{enumerate}
        \item $H$ is controllable,
        \item $H$ is $\chi$-stable,
        \item $H$ is $\chi$-semistable.
    \end{enumerate}
\end{prop}
\begin{proof}
    This is Bader's Proposition 3.7 applied to $n = 2$, $p = 1$, so that $\min\{p, n\} = 1$.
\end{proof}

\begin{rmk}
    The above proposition describes the open chamber in which controllability, stability, and semistability coincide. In our present example, however, we deliberately choose the boundary character $\rho = (-2, 3, 1)$, for which $2 \cdot (-2) + 3 + 1 = 0$. Hence $\rho$ lies on the boundary of the chamber in Proposition \ref{Propo 3.7 Bader}.

    Consequently, the above proposition does not apply to $\rho$, and one can no longer expect $\rho$-semistability and $\rho$-stability to coincide. This boundary choice is what allows strictly $\rho$-semistable Helmke systems to occur.
\end{rmk}

In the case where stability does not coincide with semistability and the parameters satisfy the inequalities, the next result shows that the controllable systems lie between the stable and semistable loci. 

\begin{prop}\label{Prop Controllability Boundary case}
    Let $n, m, p \geq 1$, and let $\mathcal{H}_{n, m, p}$ be the space of Helmke systems of type $(n, m, p)$. Then the reductive group is $\mathrm{G} = \mathrm{GL}_n(\mathbb{C}) \times \mathrm{GL}_n(\mathbb{C}) \times \mathrm{GL}_p(\mathbb{C})$. For a character $\chi = (r, s, t) \in \mathbb{Z}^3$, write $\chi(g_0, g_1, g_2) = \det(g_0)^r \det(g_1)^s \det(g_2)^t$.

    Assume that $nr + (n-1)s + \min \{p, n\}t = 0$, $r + s> 0$, and $t > 0$. Let $\mathcal{H}^c_{n, m, p} \subseteq \mathcal{H}_{n, m, p}$ denote the locus of controllable Helmke systems. Then
    $$
    \mathcal{H}^{st}_{n, m, p}(\chi) \subseteq \mathcal{H}^c_{n, m, p} \subseteq \mathcal{H}^{ss}_{n, m, p}(\chi).
    $$
\end{prop}
\begin{proof}
    Note that $\mathcal{H}_{n, m, p}$ is regarded as the representation space of the Helmke quiver. The set $M$ of marked vertices consists of $v_0, v_1 \cong \mathbb{C}^n$, and $v_2 \cong \mathbb{C}^p$, and its unmarked vertices are $v_3 \cong \mathbb{C}^m$, and $v_4 \cong \mathbb{C}^p$. Consider a subrepresentation $U = (U_0, U_1, U_2, U_3, U_4)$ such that
    $E(U_0) \subseteq U_1$, $A(U_0) \subseteq U_1$, $B(U_3) \subseteq U_1$, $C(U_0) \subseteq U_2$, $D(U_3) \subseteq U_2$, and $F(U_4) \subseteq U_2$. Write $u_i = \dim_\mathbb{C} U_i$, and for the character $\chi = (r, s, t)$ we define
    $$
    \langle \chi, U \rangle_M = ru_0 + su_1 + tu_2.
    $$
    Let $\mathrm{V} = (V_0, V_1, V_2, V_3, V_4)$ be the space that each vertex is of full dimension, i.e. $V_0, V_1 \cong \mathbb{C}^n$, $V_2, V_4 \cong \mathbb{C}^p$, and $V_3 \cong \mathbb{C}^m$.
    $$
    \langle \chi, \mathrm{V}\rangle_M = rn + sn + tp = n(r + s) + pt.
    $$

    A Helmke system $H \in \mathcal{H}_{n, m, p}$ is $\chi$-semistable if and only if the following two conditions hold \cite[Corollary 1.20]{Bader2008}:
    \begin{enumerate}
        \item $\langle \chi, U \rangle_M \geq 0$ for every subrepresentation $U$ with $U_3 = U_4 = 0$, and
        \item $\langle \chi, U \rangle_M \geq \langle \chi, \mathrm{V}\rangle_M$ for every subrepresentation $U$ with $U_3 = V_3$, $U_4 = V_4$.
    \end{enumerate}
    Moreover, $H$ is $\chi$-stable if and only if the first inequality is strict for every non-zero proper subrepresentations with $U_3 = U_4 = 0$, and the second inequality is strict for every proper subrepresentation with $U_3 = V_3$ and $U_4 = V_4$.
    
    Let $N \geq 2$ be an integer, and consider $\chi_N = (Nr, Ns, Nt - 1)$ such that
    $$
    \lim_{N \to \infty} \frac{1}{N}\chi_N = \lim_{N \to \infty} \bigl( r, s, t - \frac{1}{N}\bigr) = (r, s ,t) = \chi.
    $$
    By the condition $nr + (n-1)s + \min\{p, n\}t = 0$, we have
    $$
    n(Nr) + (n-1)(Ns) + \min\{p, n\}(Nt - 1) = -\min\{p, n\} < 0.
    $$
    Also $Nr + Ns = N(r + s) > 0$, and since $t \in \mathbb{Z}$ with $t > 0$, it follows that $Nt - 1 > 0$ for every $N \geq 2$. Hence $\chi_N$ lies in the chamber of Bader's Proposition 3.7.

    We first prove $\mathcal{H}^{st}_{n, m, p} (\chi) \subseteq \mathcal{H}^c_{n, m, p}$. Let $H \in \mathcal{H}^{st}_{n, m, p}(\chi)$, and let $U$ be a subrepresentation with $U_3 = U_4 = 0$. If $U = 0$, there is nothing to check in the first family. If $U \neq 0$, then $\chi$-stability gives $\langle \chi, U \rangle_M > 0$, consequently, this implies
    $$
    \langle \chi, U\rangle_M \geq 1.
    $$
    Note also that
    $$
    \langle \chi_N, U \rangle_M = Nru_0 + Nsu_1 + (Nt-1)u_2 = N\langle\chi, U\rangle_M - u_2.
    $$
    Since $0 \leq u_2 \leq p$, we get
    $$
    \langle \chi_N, U \rangle_M \geq N-p.
    $$
    Thus $\langle \chi_N, U \rangle_M > 0$ for $N > p$. Hence the strict $\chi_N$-stability inequality holds for all non-zero subrepresentations with $U_3 = U_4 = 0$.

    Now let $U$ be a proper subrepresentation with $U_3 = V_3$, and $U_4 = V_4$. Since $H$ is $\chi$-stable, we have $\langle \chi, U \rangle_M > \langle \chi, \mathrm{V}\rangle_M$. Note that both sides are integers, hence
    $$
    \langle \chi, U \rangle_M - \langle \chi, \mathrm{V}\rangle_M \geq 1.
    $$
    Moreover, $\langle \chi_N, \mathrm{V}\rangle_M = N\langle \chi, \mathrm{V}\rangle_M - p$ because the marked vertex $v_2$ has dimension $p$. Therefore,
    $$
    \begin{aligned}
        \langle \chi_N, U\rangle_M - \langle \chi_N, \mathrm{V}\rangle_M &= \bigl(N\langle \chi, U \rangle_M - u_2\bigr) - \bigl(N\langle \chi, \mathrm{V}\rangle_M - p\bigr)\\
        &= N\bigl(\langle \chi, U \rangle_M - \langle \chi, \mathrm{V}\rangle_M\bigr) + \bigl(p - u_2\bigr).
    \end{aligned}
    $$
    Since $\langle \chi, U \rangle_M - \langle \chi, \mathrm{V}\rangle_M \geq 1$ and $p - u_2 \geq 0$, it follows that
    $$
    \langle \chi_N, U \rangle_M - \langle \chi_N, \mathrm{V}\rangle_M > 0.
    $$
    Thus $\langle \chi_N, U \rangle_M > \langle \chi_N, \mathrm{V}\rangle_M$. Hence $H$ is $\chi_N$-stable for all sufficiently large $N$, and by Bader's Proposition 3.7, $H$ is controllable. Therefore, $\mathcal{H}^{st}_{n, m, p}(\chi) \subseteq \mathcal{H}^c_{n, m, p}$.

    We now prove the second inclusion $\mathcal{H}^c_{n, m, p} \subseteq \mathcal{H}^{ss}_{n, m, p}(\chi)$. Let $H \in \mathcal{H}^c_{n, m, p}$. Since $\chi_N$ lies in the chamber defined in Proposition 3.7 of \cite{Bader2008}, it follows that $H$ is $\chi_N$-semistable for every $N \geq 2$. Let $U$ be a subrepresentation with $U_3 = U_4 = 0$. Since $H$ is $\chi_N$-semistable, we have
    $$
    \langle \chi_N, U \rangle_M = N\langle\chi, U\rangle_M - u_2 \geq 0
    $$
    for every $N \geq 2$. Suppose, for contradiction, that $\langle \chi, U \rangle_M < 0$. Since $\langle \chi, U \rangle_M \in \mathbb{Z}$, this means $\langle \chi, U \rangle_M \leq -1$. Hence
    $$
    N\langle\chi, U\rangle_M - u_2 \leq -N-u_2 < 0,
    $$
    which contradicts $\chi_N$-semistability. Therefore,
    $$
    \langle \chi, U \rangle_M \geq 0.
    $$

    Now let $U_3 = V_3$, and $U_4 = V_4$. Similarly, we have, for the whole Helmke representation $\mathrm{V}$, that
    $$
    \langle \chi_N, \mathrm{V}\rangle_M = N\langle \chi, \mathrm{V}\rangle_M - p.
    $$
    Since $H$ is $\chi_N$-semistable, we have $\langle \chi_N, U\rangle_M \geq \langle\chi_N, \mathrm{V}\rangle_M$, which gives
    $$
    N(\langle \chi, U \rangle_M - \langle \chi, \mathrm{V}\rangle_M) + (p - u_2) \geq 0.
    $$
    Suppose, for contradiction, that $\langle \chi, U \rangle_M < \langle \chi, \mathrm{V}\rangle_M$. Then this inequality implies
    $$
    \langle \chi, U \rangle_M - \langle \chi, \mathrm{V}\rangle_M \leq -1.
    $$
    Therefore,
    $$
    N(\langle \chi, U \rangle_M - \langle \chi, \mathrm{V}\rangle_M) + (p - u_2) \leq -N + (p - u_2).
    $$
    For $N > p$, the right hand side is strictly negative, contradicting $\chi_N$-semistability. Thus $\langle \chi, U \rangle_M \geq \langle \chi, \mathrm{V}\rangle_M$. Since $H \in \mathcal{H}^c_{n, m, p}$ was arbitrary, this proves $\mathcal{H}^c_{n, m, p} \subseteq \mathcal{H}^{ss}_{n, m, p}(\chi)$.
\end{proof}

\begin{rmk}
    We note that the choice $\chi_N = (Nr, Ns, Nt-1) = N\cdot\chi - (0,0,1)$ is not unique. The idea of making such a choice is to perturb the boundary character $\chi$ into the chamber defined in Bader's Proposition 3.7. More generally, let $\delta = (a, b, c) \in \mathbb{Z}^3$ satisfy $na + (n-1)b + \min\{p, n\}c > 0$. Then the characters $\chi_N = N\chi - \delta$ satisfy $-\bigl(na + (n-1)b + \min\{p, n\}c \bigr) < 0$.

    Moreover, since $r+s > 0$ and $t > 0$, we have $N(r+s) - (a+b) > 0$ and $Nt - c > 0$ for all sufficiently large $N$. Hence $\chi_N$ lies in the chamber defined in Proposition \ref{Propo 3.7 Bader} if $N$ is sufficiently large. The proof of the proposition also works with this more general choice of $\chi_N$. In particular, for any subrepresentation $U$, we have
    $$
    \langle \chi_N, U \rangle_M = N\langle\chi, U \rangle_M - \langle \delta, U\rangle_M.
    $$
    The term $\langle \delta, U\rangle_M$ is uniformly bounded because the dimensions $u_0, u_1, u_2$ satisfy $0\leq u_0 \leq n$, $0\leq u_1 \leq n$, and $0\leq u_2\leq p$. Consequently, as $N \to \infty$, the leading term $N\langle \chi, U \rangle_M$ dominates the perturbation.
\end{rmk}

\begin{prop}
    Let $V= \mathcal{H}_{2, 3, 1}, \mathrm{G} = \mathrm{GL}_2(\mathbb{C})\times \mathrm{GL}_2(\mathbb{C}) \times \mathrm{GL}_1(\mathbb{C})$, and define the character
    $$
    \rho(g_0, g_1, g_2) = \det(g_0)^{-2}\det(g_1)^3\det(g_2).
    $$
    Let $M = \{v_0, v_1, v_2\}$ be the marked vertices. For a subrepresentation
    $$
    U = (U_0, U_1, U_2, U_3, U_4)
    $$
    write $u_i :=  \dim_\mathbb{C}U_i$ and define
    $$
    \langle\rho, U\rangle_M := -2u_0 + 3u_1 + u_2.
    $$
    Then $\langle \rho, \mathrm{V}\rangle_M = -2 \cdot 2 + 3\cdot 2 + 1 = 3$. A Helmke system $H \in V$ is $\rho$-semistable if and only if $\langle \rho, U\rangle_M \geq 0$ for every subrepresentation $U$ with $U_3 = U_4 = 0$, and $\langle \rho, U\rangle_M \geq 3$ for every subrepresentation $U$ with $U_3 = \mathbb{C}^3$ and $U_4 = \mathbb{C}$.

    Moreover, $H$ is $\rho$-stable if and only if these inequalities are strict for every non-zero proper subrepresentation in the first family, and strict for every proper subrepresentation in the second family.
\end{prop}
\begin{proof}
    This is Corollary 1.20 of \cite{Bader2008} applied to the marked set $M = \{v_0, v_1, v_2\}$. The corollary says that for quivers with marked vertices, one need only to test subrepresentations whose unmarked vertices are all zero or all full. In this example, the unmarked vertices are $v_3$ and $v_4$, with dimensions $3$ and $1$, so the two cases are
    $$
    (U_3, U_4) = (0, 0) \text{ or } (U_3, U_4) = (\mathbb{C}^3, \mathbb{C}).
    $$
    The first case gives the inequality $\langle \rho, U\rangle_M \geq 0$, and the second case gives $\langle \rho, U\rangle_M \geq \langle \rho, \mathrm{V}\rangle_M = 3$.

    The stable case is obtained by requiring strict inequality for non-zero proper subrepresentations, which is also stated in Corollary 1.20 of \cite{Bader2008}.
\end{proof}

By the proposition above, to obtain $\rho$-(semi)stability it is enough to consider proper non-zero subrepresentations $U = (U_0, U_1, U_2, U_3, U_4)$ for which either $U_3 = U_4 =0$ or $U_3 = \mathbb{C}^3, U_4 = \mathbb{C}$. Here $U_0, U_1 \subseteq \mathbb{C}^2$, $U_2, U_4 \subseteq \mathbb{C}$, $U_3 \subseteq \mathbb{C}^3$, and the subrepresentation conditions are
$$
E(U_0), A(U_0), B(U_3)\subseteq U_1,\quad C(U_0), D(U_3), F(U_4) \subseteq U_2.
$$

We now construct a strictly $\rho$-semistable point of the space $\mathcal{H}_{2, 3, 1}$. For this it is sufficient to produce a proper subrepresentation $U$ in the first family, namely with $U_3 = U_4 = 0$ such that $\langle \rho, U\rangle_M = 0$. Let $u_i:=\dim_\mathbb{C}U_i$. We can check that the only non-zero marked dimension vector satisfying $-2\cdot u_0 + 3\cdot u_1 + u_2 = 0$ is
$$
(u_0, u_1, u_2) = (2, 1, 1).
$$
Hence we look for a subrepresentation of the form
$$
U = (\mathbb{C}^2, L, \mathbb{C}, 0, 0),
$$
where $L \subseteq \mathbb{C}^2_{v_1}$ is a line. This requires
$$
E(U_0) \subseteq U_1, A(U_0)\subseteq U_1,
$$
and the condition $C(U_0)\subseteq U_2$ is automatic because $U_2 = \mathbb{C}$.

On the other hand, if there are subrepresentations having a non-zero subspace $U_0\subseteq \mathbb{C}^2$ with $U_1 = U_2 = 0$, then such a subrepresentation would satisfy $\langle\rho, U\rangle_M = -2u_0 < 0$, violating the semistability. Hence we need to ensure that a non-zero $U_0$ does not occur with $U_1 = U_2 = 0$, so we impose the condition $\ker E \cap \ker A = \{0\}$.

Secondly, we consider subrepresentations of the form $U = (U_0, U_1, U_2, \mathbb{C}^3, \mathbb{C})$. In this case, the arrows $B, D$, and $F$ impose the additional conditions
$$
B(\mathbb{C}^3) \subseteq U_1,\quad D(\mathbb{C}^3)\subseteq U_2,\quad F(\mathbb{C})\subseteq U_2.
$$
Consequently, to ensure the semistability inequality for every such subrepresentation, it is sufficient to choose the map $B$ surjective and the map $F$ non-zero. These two conditions imply that $U_1 = \mathbb{C}^2$ and $U_2 = \mathbb{C}$.

Let $(e_1, e_2)$ be a basis of $\mathbb{C}^2$ at the vertex $v_0$, and $(f_1, f_2)$ of $\mathbb{C}^2$ at the vertex $v_1$. Set $L = \operatorname{span}(f_2) \subset \mathbb{C}^2$ at $v_1$, and then choose $E, A : \mathbb{C}^2_{v_0} \to \mathbb{C}^2_{v_1}$ so that $\mathrm{im} E = \mathrm{im} A = L, \ker E = \mathrm{span}(e_2), \text{ and } \ker A = \mathrm{span}(e_1)$. Then we may take, with respect to the chosen bases, $E = \begin{pmatrix}
    0 & 0\\ 1 & 0
\end{pmatrix}, A = \begin{pmatrix}
    0 & 0\\ 0 & 1
\end{pmatrix}$. We further choose the map $B: \mathbb{C}^3 \to \mathbb{C}^2$ to be surjective, $C$ to be non-zero, $F: \mathbb{C} \to \mathbb{C}$ to be the identity map on $\mathbb{C}$, and $D$ arbitrary. Without loss of generality, we take
$$
B = \begin{pmatrix}
    1 & 0 & 0\\ 0 & 1 & 0
\end{pmatrix},\quad C = \begin{pmatrix}
    1 & 1
\end{pmatrix},\quad D = \begin{pmatrix}
    0 & 0 & 1
\end{pmatrix}.$$

\begin{lemma}\label{lemma 1}
    The resulting Helmke system $H_{ss} = (E, A, B, C, D, F)$ is strictly $\rho$-semistable. Consequently, the space $\mathcal{H}_{2, 3, 1}$ has non-empty strictly $\rho$-semistable locus under the action of $\mathrm{GL}_{2,2,1}$.
\end{lemma}
\begin{proof}[Proof of Lemma \ref{lemma 1}]
    Let $U = (U_0, U_1, U_2, 0, 0)$ be a proper non-zero subrepresentation. If $\dim_\mathbb{C} U_0 = 0$, then
    $$
    \langle\rho, U\rangle_M = 3\dim_\mathbb{C}U_1 + \dim_\mathbb{C}U_2 \geq 0.
    $$
    Now suppose $\dim_\mathbb{C} U_0 > 0$. Because $\ker E \cap \ker A = \{0\}$, there is no non-zero subspace $U_0$ on which both $E$ and $A$ vanish. Hence $\dim_\mathbb{C} U_1\geq 1$.

    If $\dim_\mathbb{C}U_0 = 1$, then
    $$
    \langle \rho, U\rangle_M = -2 + 3\dim_\mathbb{C}U_1 +\dim_\mathbb{C}U_2 \geq -2 + 3 = 1 > 0.
    $$
    If $\dim_\mathbb{C}U_0 = 2$, then $U_0 = \mathbb{C}^2$. Since $\operatorname{im} E = \operatorname{im} A = L$, any subspace $U_1$ that satisfies $E(U_0), A(U_0) \subseteq U_1$ contains $L$. Thus either $U_1 = L$ or $U_1 = \mathbb{C}^2$.
    
    If $\dim_\mathbb{C}U_1 = 1$, then $U_1 = L$. Since $C$ is non-zero and $U_0 = \mathbb{C}^2$, we have $C(U_0) = \mathbb{C}$, so $U_2 = \mathbb{C}$. Therefore
    $$
    \langle\rho, U\rangle_M = -4+3+1 = 0.
    $$
    If $\dim_\mathbb{C}U_1 = 2$, then
    $$
    \langle \rho, U\rangle_M = -4 + 6 + \dim_\mathbb{C}U_2 > 0.
    $$
    Hence every proper non-zero subrepresentation with vertices $v_3 = v_4 = 0$ satisfies the inequality $\langle \rho, U\rangle_M \geq 0$.

    Now we consider a subrepresentation of the second family, namely $U = (U_0, U_1, U_2, \mathbb{C}^3, \mathbb{C})$. Since $B$ is surjective and $F$ is non-zero, we have $U_1 = \mathbb{C}^2, U_2 = \mathbb{C}$. Thus
    $$
    \langle \rho, U\rangle_M = -2\dim_\mathbb{C}U_0 + 6 + 1 = 7 - 2\dim_\mathbb{C}U_0 \geq 3.
    $$

    Therefore, we conclude that $H_{ss}$ is $\rho$-semistable. It is not $\rho$-stable because there is a proper subrepresentation $U = (\mathbb{C}^2, L, \mathbb{C}, 0, 0)$ that define a subrepresentation satisfying $\langle \rho, U \rangle_M = -4 + 3 + 1 = 0$. Consequently, the $\rho$-semistability is strict.
\end{proof}

\begin{lemma}\label{lemma 2}
    The $\rho$-stable locus of the space $\mathcal{H}_{2, 3, 1}$ is non-empty.
\end{lemma}
\begin{proof}[Proof of Lemma \ref{lemma 2}]
    It is sufficient to construct one point for which the semistability inequalities $\langle \rho, U\rangle_M$ are all strict. From the preceding proof, the only possible equality case is $(\dim_\mathbb{C}U_0, \dim_\mathbb{C}U_1, \dim_\mathbb{C}U_2) = (2 ,1 ,1)$. This can be excluded by choosing $E$ invertible, which means $E(U_0) \subseteq U_1$. Thus $\dim_\mathbb{C}U_1 \geq \dim_\mathbb{C}U_0$. In particular, if $U_0 = \mathbb{C}^2$, then $U_1 = \mathbb{C}^2$.

    Choose $E = \begin{pmatrix}
        1 & 0\\ 0 & 1
    \end{pmatrix}$, let $A$ be the zero map, $B = \begin{pmatrix}
        1 & 0 & 0 \\ 0 & 1 & 0
    \end{pmatrix}, C = \begin{pmatrix}
        1 & 0
    \end{pmatrix}, D = \begin{pmatrix}
        0 & 0 & 1
    \end{pmatrix},$ and $F: \mathbb{C} \to \mathbb{C}$ is still assumed to be the identity map. We show that the Helmke system $H_{st}$ formed under these choices is stable.

    Let $U = (U_0, U_1, U_2, 0, 0)$ be a proper non-zero subrepresentation. If $\dim_\mathbb{C}U_0 = 0$, then $\langle \rho, U \rangle_M = 3\dim_\mathbb{C}U_1 + \dim_\mathbb{C}U_2 > 0$. If $\dim_\mathbb{C}U_0 = 1$, then the invertibility of $E$ would give $U_0 = E(U_0) \subseteq U_1$, which means $\dim_\mathbb{C}U_1\geq 1$. Thus $\langle \rho, U\rangle_M \geq -2 + 3 + \dim_\mathbb{C}U_2 \geq 1$. If $\dim_\mathbb{C}U_0 = 2$, then $U_0 = U_1 = \mathbb{C}^2$. Thus $\langle \rho, U \rangle_M = -4 + 6 + \dim_\mathbb{C}U_2 = 2 + \dim_\mathbb{C}U_2 > 0$.

    Now consider $U = (U_0, U_1, U_2, \mathbb{C}^3, \mathbb{C})$. Because the map $B$ is surjective and the map $F$ is not the zero map, we have $U_1 = \mathbb{C}^2, U_2 = \mathbb{C}$. So $\langle \rho, U\rangle_M = 7-2\dim_\mathbb{C}U_0$, and the properness of $U$ implies that $\dim_\mathbb{C} U_0 < 2$. Hence $\langle \rho, U\rangle_M > 3$.

    We have shown that all the inequalities are strict under the choice of the Helmke system $H_{st}$. As a consequence, the system $H_{st}$ is a stable point in the space $\mathcal{H}_{2, 3, 1}$, and thus the stable locus is not empty. 
\end{proof}

Recall that the reductive group is $\mathrm{G} = \mathrm{GL}_2(\mathbb{C}) \times \mathrm{GL}_2(\mathbb{C}) \times \mathrm{GL}_1(\mathbb{C})$, we consider a 1-PS $\lambda : \mathbb{C}^* \to \mathrm{GL}_2(\mathbb{C}) \times \mathrm{GL}_2(\mathbb{C}) \times \mathrm{GL}_1(\mathbb{C})$ of the form
$$
\lambda(t) = (\mathrm{diag}(t^{a_1}, t^{a_2}), \mathrm{diag}(t^{b_1}, t^{b_2}), t^c),
$$
where $a_i, b_j, c \in \mathbb{Z}$ with $i,j \in \{1, 2\}$. We may assume $a_1 \leq a_2, b_1 \leq b_2$ so that the same conjugacy class would not be counted twice.

Recall also that the affine space $V = \mathcal{H}_{2, 3, 1}$ that the group $\mathrm{G}$ acts on is
$$
\mathcal{H}_{2, 3, 1} = M_{2\times 2}(\mathbb{C})^2 \oplus M_{2 \times 3}(\mathbb{C}) \oplus M_{1 \times 2}(\mathbb{C}) \oplus M_{1 \times 3}(\mathbb{C}) \oplus M_{1\times 1}(\mathbb{C}),
$$
where the action is
$$
(g_0, g_1, g_2) \cdot (E, A, B, C, D, F) = (g_1Eg_0^{-1}, g_1Ag_0^{-1}, g_1B, g_2Cg_0^{-1}, g_2D, g_2F).
$$
To set up the transversality framework, we need to compute the complex dimension $m_\lambda = \dim_\mathbb{C}V(\lambda)_-$ of the negative space under this action, and the complex dimension $\dim_\mathbb{C}(\mathrm{G} \cdot \lambda)$ of the orbit. To apply Theorem \ref{Thm 2.6}, we need to find the minimum of the difference $2m - 2\dim_\mathbb{C}(\mathrm{G} \cdot \lambda)$ over all conjugacy classes $[\lambda_j]$. We treat the computations in two cases, one is when the unmarked vertices are $0$, and one is when the unmarked vertices have full dimension.

We begin with the family of subrepresentations for which the unmarked vertices are zero. For $E\in M_{2\times 2}(\mathbb{C})$, let $e_{ji} \in M_{2 \times 2}(\mathbb{C})$ be the matrix with a $1$ in row $j$, column $i$, and zero elsewhere. Since $E$ transforms by $E \mapsto g_1Eg_0^{-1}$, we obtain
$$
e_{ji} \mapsto \mathrm{diag}(t^{b_1}, t^{b_2})e_{ji}\mathrm{diag}(t^{-a_1}, t^{-a_2}) = t^{b_j - a_i}e_{ji}.
$$
Thus the weight of the basis vector $e_{ji}$ is $b_j - a_i$. The same calculation applies to $A$, since $A$ transforms in exactly the same way. Therefore, the contribution to the dimension of the negative weight space from $E, A$ is
$$
2 \cdot \#\{(j, i) \mid b_j - a_i < 0\}.
$$

Next consider $B\in M_{2\times 3}(\mathbb{C})$ with the action $B \mapsto g_1B$, let $b_{ji}$ be the matrix unit. Then
$$
b_{jl}\mapsto \mathrm{diag}(t^{b_1}, t^{b_2}) b_{jl} = t^{b_j}b_{jl}.
$$
Since there are $3$ columns, the contribution from $B$ is
$$
3 \cdot \#\{j \mid b_j < 0\}.
$$
Now consider $C \in M_{1\times 2}(\mathbb{C})$, which transforms by $C\mapsto g_2Cg_0^{-1}$. Its two basis vectors therefore have weights $c - a_1$ and $c - a_2$. Thus the contribution from $C$ is
$$
\#\{i \mid  c-a_i < 0 \}.
$$
Finally, $D \in M_{1 \times 3}(\mathbb{C})$ and $F \in M_{1 \times 1}(\mathbb{C})$ both transform by multiplication with $g_2$, so every basis vector in $D$ and $F$ has weight $c$. Since $D$ contributes three coordinates and $F$ contributes one, the contribution from $D$ and $F$ is $4$ if $c < 0$ and $0$ otherwise.

Therefore, collecting these terms, we obtain
$$
m_\lambda = 2 \cdot \#\{(j, i) \mid b_j - a_i < 0\} + 3 \cdot \#\{j \mid b_j < 0\} + \#\{i | c-a_i < 0 \} + \begin{cases}
    4, c < 0\\
    0, \text{otherwise}
\end{cases}.
$$

We next compute the orbit dimension $\dim_{\mathbb{C}}(\mathrm{G} \cdot \lambda)$ Note that $\dim_\mathbb{C}(\mathrm{G} \cdot \lambda) = \dim_\mathbb{C}\mathrm{G} - \dim_\mathbb{C}C_\mathrm{G}(\lambda)$, where $C_\mathrm{G}(\lambda)$ is the centraliser.

Because $\mathrm{G}$ is a direct product, conjugation acts independently on the three factors, i.e.
$$(g\cdot \lambda)(t) = (g_0\lambda_0(t)g_0^{-1}, g_1\lambda_1(t)g_1^{-1}, g_2\lambda_2(t)g_2^{-1}).
$$
Observe that $g_2 \in \mathrm{GL}_1(\mathbb{C})$, so $g_2\lambda_2(t)g_2^{-1}$ is the usual multiplication of non-zero complex numbers, which is commutative. Hence every element in $\mathrm{GL}_1(\mathbb{C})$ lies in the centre, i.e. $Z(\mathrm{GL}_1(\mathbb{C})) = \mathrm{GL}_1(\mathbb{C})$. Given that $\lambda_2(t) = t^c$, thus the $\mathrm{GL}_1(\mathbb{C})$-factor has no impact on the orbit dimension.

For the first $\mathrm{GL}_2(\mathbb{C})$-factor, we have $\lambda_0(t) = \mathrm{diag}(t^{a_1}, t^{a_2})$. If $a_1 = a_2$, then $\mathrm{diag}(t^{a_1}, t^{a_2}) = t^{a_1}I_2$, which commutes with every element of $\mathrm{GL}_2(\mathbb{C})$. If $a_1 < a_2$, let $N = \begin{pmatrix}
    x & y \\ z & w
\end{pmatrix} \in \mathrm{GL}_2(\mathbb{C})$, then $N\lambda_{0}(t) = \lambda_0(t)N$ implies that $y(t^{a_2} - t^{a_1}) = 0$, and $z(t^{a_1} - t^{a_2}) = 0$ for all $t \in \mathbb{C}^*$. Since $a_1 \neq a_2$, the two equalities holds if and only if $y = z= 0$. So the centraliser of this action is isomorphic to the torus $(\mathbb{C}^*)^2$. Therefore, we obtain the orbit dimension $\dim_\mathbb{C}(\mathrm{G} \cdot \lambda) = \dim_\mathbb{C}\mathrm{GL}_2(\mathbb{C}) - \dim_\mathbb{C}C_{\mathrm{GL}_2(\mathbb{C})}(\lambda_0) = 4 - 2 = 2$. A similar argument holds for the second $\mathrm{GL}_2(\mathbb{C})$-factor. Therefore, we conclude that
$$
\dim_\mathbb{C}(\mathrm{G}\cdot \lambda) = \begin{cases}
    0, a_1 = a_2\\
    2, a_1< a_2
\end{cases} + \begin{cases}
    0, b_1 = b_2\\
    2, b_1 < b_2
\end{cases}.
$$

For the family of the Helmke systems that have subrepresentations whose unmarked vertices are zero, destabilising or strictly semistable subrepresentations in this family have marked dimension triples satisfying $(\dim U_0, \dim U_1, \dim U_2)$ should satisfy $-2\dim U_0 + 3\dim U_1 + \dim U_2 \leq 0$. So we have the following table
\begin{table}[H]
\centering
\begin{tabular}{|c|c|c|c|}
\hline
 $(\dim U_0, \dim U_1, \dim U_2)$& $m_\lambda = \dim_\mathbb{C}V(\lambda)_-$ & $\dim_\mathbb{C}(\mathrm{G}\cdot \lambda)$ & $2m_\lambda - 2\dim_\mathbb{C}(\mathrm{G}\cdot \lambda)$ \\ \hline
$(1, 0, 0)$ & $5$ & $2$ & $6$ \\ \hline
$(1, 0, 1)$ & $4$ & $2$ & $4$ \\ \hline
$(2, 0, 0)$ & $10$ & $0$ & $20$ \\ \hline
$(2, 0, 1)$ & $8$ & $0$ & $16$ \\ \hline
$(2, 1, 0)$ & $6$ & $2$ & $8$ \\ \hline
$(2, 1, 1)$ & $4$ & $2$ & $4$ \\ \hline
\end{tabular}
\end{table}

Now we consider the case where unmarked vertices are of full dimension, namely those of the form $U = (U_0, U_1, U_2, \mathbb{C}^3, \mathbb{C})$. The subrepresentation conditions are then
$$
E(U_0)\subseteq U_1, A(U_0)\subseteq U_1, B(\mathbb{C}^3)\subseteq U_1, C(U_0) \subseteq U_2, D(\mathbb{C}^3) \subseteq U_2, F(\mathbb{C})\subseteq U_2.
$$
Concretely, we choose complements of $U_0$ and $U_1$ in the marked vertex spaces
$$
\mathbb{C}^2_{v_0} = U_0^\perp \oplus U_0,\quad \mathbb{C}^2_{v_1} = U_1^\perp \oplus U_1,
$$
with $\dim U_0^\perp = 2 - \dim U_0, \text{ and } \dim U_1^\perp = 2-\dim U_1$. Since $\dim\mathbb{C}_{v_2} = 1$, we have only two possibilities for $U_2$, namely $\dim U_2 = 0$ or $\dim U_2 = 1$.

With respect to these decompositions, the maps $E, A : \mathbb{C}^2_{v_0} \to \mathbb{C}_{v_1}^2$ may be written in block form as
$$
E = \begin{pmatrix}
    E_{00} & E_{01}\\
    E_{10} & E_{11}
\end{pmatrix},\quad A = \begin{pmatrix}
    A_{00} & A_{01}\\
    A_{10} & A_{11}
\end{pmatrix},
$$
where the column decomposition is $U_0^\perp \oplus U_0$, and the row decomposition is $U_1^\perp \oplus U_1$. Thus the block $E_{01}$ is the component $U_0 \to U_1^\perp$, and similarly for $A_{01}$.

Likewise, the map $B: \mathbb{C}^3 \to \mathbb{C}^2$ decomposes as
$$
B = \begin{pmatrix}
    B_0\\
    B_1
\end{pmatrix},
$$
where $B_0 : \mathbb{C}^3 \to U_1^\perp$, and the map $C : \mathbb{C}^2 \to \mathbb{C}$ decomposes as $C = \begin{pmatrix}
    C_0 & C_1
\end{pmatrix}$, where $C_1 : U_0 \to \mathbb{C}$.

The subrepresentation conditions of each arrow on each $U_i$ then give the total number of the vanishing entries of these matrices. In particular, we can define a 1-PS so that these entries are the elements in the negative weight space. Consequently, the dimension of the negative weight space coincides with the total number of the vanishing entries, thus
$$
m_\lambda = \dim_\mathbb{C}V(\lambda)_- = 2\dim U_0(2 - \dim U_1) + 3(2-\dim U_1) + \dim U_0(1-\dim U_2) + 4(1-\dim U_2).
$$

We now need to find such a 1-PS $\lambda: \mathbb{C}^* \to \mathrm{GL}_2(\mathbb{C}) \times \mathrm{GL}_2(\mathbb{C}) \times \mathrm{GL}_1(\mathbb{C})$, i.e. its entries in $E_{01}, A_{01}: U_0 \to U_1^\perp$ have negative weight, its entries in $B_0: \mathbb{C}^3 \to U_1^\perp$ have negative weight, its entries in $C_1: U_0 \to \mathbb{C}$ have negative weight given that $\dim U_2 = 0$, its entries in $D, F$ also have negative weight when $\dim U_2 = 0$. For $B \mapsto g_1B$, we want the rows mapped to $U_1^\perp$ have negative weight and the rows mapped to $U_1$ have non-negative weight, so without loss of generality we can choose weights on $U_1^\perp$ to be $-1$ and weights on $U_1$ to be $0$. Thus the second $\mathrm{GL}_2(\mathbb{C})$-factor of the 1-PS $\lambda$ becomes
$$
\operatorname{diag}\bigl(\underbrace{t^{-1},\ldots ,t^{-1}}_{(2-\dim U_1) \text{ terms}}, \underbrace{1, \ldots, 1}_{(\dim U_1) \text{ terms}} \bigr).
$$
For $E \mapsto g_1Eg_0^{-1}$ (a similar argument holds for $A$), we want the block $U_0 \to U_1^\perp$ to have negative weight, so the weight of $U_0$ should satisfy $-1 - \text{weight}(U_0) < 0$, then we can choose the weight of $U_0$ to be $0$. Thus the first $\mathrm{GL}_2(\mathbb{C})$-factor is
$$
\operatorname{diag}\bigl(\underbrace{t^{-1}, \ldots, t^{-1}}_{(2-\dim U_0) \text{ terms}}, \underbrace{1, \ldots, 1}_{(\dim U_0) \text{ terms}} \bigr).
$$
For the $\mathrm{GL}_1(\mathbb{C})$-factor, if $U_2 = \mathbb{C}$, then $C(\mathbb{C}) \subseteq \mathbb{C}$, which means no entries in $C, D, F$ have negative weight, so a natural choice of weight is $0$. If $U_2 = 0$, and we need $C_1: U_0 \to \mathbb{C}$ to have negative weight, and all entries of $D, F$ mapped to $U_2^\perp = \mathbb{C}$ also have negative weight. So we choose this factor to be $t^{\dim U_2 - 1}$.

Therefore, we obtain the expression
$$
\lambda(t) = \bigl(\operatorname{diag}\bigl(\underbrace{t^{-1},\ldots ,t^{-1}}_{(2-\dim U_0) \text{ terms}}, \underbrace{1, \ldots, 1}_{(\dim U_0) \text{ terms}} \bigr), \operatorname{diag}\bigl(\underbrace{t^{-1}, \ldots, t^{-1}}_{(2-\dim U_1) \text{ terms}}, \underbrace{1, \ldots, 1}_{(\dim U_1) \text{ terms}} \bigr), t^{\dim U_2 - 1} \bigr).
$$
Since $\dim U_0 \in \{0, 1, 2\}, \dim U_1 \in \{0, 1, 2\}$, and $\dim U_2\in \{0, 1\}$, it follows that
\begin{itemize}
    \item if $\dim U_0 = 0$, then the first factor is $t^{-1}I_2$,
    \item if $\dim U_0 = 1$, then the first factor is $\operatorname{diag}(t^{-1}, 1)$,
    \item if $\dim U_0 = 2$, then the first factor is $I_2$.
\end{itemize}
This same observation applies to the second $\mathrm{GL}_2$-factor. Hence the orbit dimension is
$$
\dim_\mathbb{C}(\mathrm{G} \cdot \lambda) = \begin{cases}
    2, \dim U_0 = 1\\
    0, \dim U_0 = 0, 2
\end{cases} + \begin{cases}
    2, \dim U_1= 1\\
    0, \dim U_1 = 0, 2
\end{cases}.
$$
Note that the possible triples $(\dim U_0, \dim U_1, \dim U_2)$ satisfy the inequality $-2\dim U_0 + 3\dim U_1 + \dim U_2 \leq 3$, thus we have the following table
\begin{table}[H]
\centering
\begin{tabular}{|c|c|c|c|}
\hline
 $(\dim U_0, \dim U_1, \dim U_2)$& $m_\lambda = \dim_\mathbb{C}V(\lambda_-)$ & $\dim_\mathbb{C}(\mathrm{G}\cdot \lambda)$ & $2m_\lambda - 2\dim_\mathbb{C}(\mathrm{G}\cdot \lambda)$ \\ \hline
$(0, 0, 0)$ & $10$ & $0$ & $20$ \\ \hline
$(0, 0, 1)$ & $6$ & $0$ & $12$ \\ \hline
$(0, 1, 0)$ & $7$ & $2$ & $10$ \\ \hline
$(1, 0, 0)$ & $15$ & $2$ & $26$ \\ \hline
$(1, 0, 1)$ & $10$ & $2$ & $16$ \\ \hline
$(1, 1, 0)$ & $10$ & $4$ & $12$ \\ \hline
$(1, 1, 1)$ & $5$ & $4$ & $2$ \\ \hline
$(2, 0, 0)$ & $20$ & $0$ & $40$ \\ \hline
$(2, 0, 1)$ & $14$ & $0$ & $28$ \\ \hline
$(2, 1, 0)$ & $13$ & $2$ & $22$ \\ \hline
$(2, 1, 1)$ & $7$ & $2$ & $10$ \\ \hline
$(2, 2, 0)$ & $6$ & $0$ & $12$ \\ \hline
\end{tabular}
\end{table}

Combining the two families, we obtain
$$
d_{\min} = \min(2m_\lambda - 2\dim_\mathbb{C}(\mathrm{G}\cdot \lambda)) = 2.
$$
Since Lemma \ref{lemma 2} shows that $V^{st}(\rho) \neq \varnothing$, $V^{st}(\rho)$ is $(2-2)$-connected means that $V^{st}(\rho)$ is path-connected.

\begin{rmk}
    Recall that Proposition \ref{Prop Controllability Boundary case} gives the inclusions
    $$
    \mathcal{H}^{st}_{n, m, p}(\chi) \subseteq \mathcal{H}^c_{n, m, p} \subseteq \mathcal{H}^{ss}_{n, m, p}(\chi)
    $$
    for characters $\chi = (r, s, t)$ satisfying $nr+(n-1)s + \min \{p, n\}t = 0$, $r + s > 0$, and $t > 0$.
    We now show, in the case $\rho = (-2, 3, 1)$ for Helmke systems of type $(2, 3, 1)$, that both inclusions can be strict.
    
    We begin by constructing a controllable Helmke system which is $\rho$-semistable but not $\rho$-stable. Let $v_0 = v_1 = \mathbb{C}^2$, $v_2 = v_4 = \mathbb{C}$, $v_3 = \mathbb{C}^3$, and consider the Helmke system $H = (E, A, B, C ,D ,F)$ given by
    $$
    E = \begin{pmatrix}
        0 & 1\\
        1 & 0
    \end{pmatrix},\quad
    A = \begin{pmatrix}
        0 & 0\\
        0 & 1
    \end{pmatrix},\quad
    B = \begin{pmatrix}
        1 & 0 & 0\\
        0 & 0 & 0
    \end{pmatrix},\quad
    C = \begin{pmatrix}
        1 & 0
    \end{pmatrix},\quad
    D = F = 0.
    $$
    We check that $H$ is controllable. For $(\alpha, \beta) \in \mathbb{C}^2$, we have
    $$
    \alpha E + \beta A = \begin{pmatrix}
        0 & \alpha\\
        \alpha & \beta
    \end{pmatrix},
    $$
    hence $\det (\alpha E + \beta A) = -\alpha^2$. Thus the first controllability condition in Definition 3.2 of \cite{Bader2008} holds. Moreover, if $\alpha \neq 0$ then the matrix $\alpha E + \beta A$ has rank $2$, and thus
    $$
    \operatorname{rank} \begin{pmatrix}
        \alpha E + \beta A & B
    \end{pmatrix}
    = 2.
    $$
    If $\alpha = 0$ and $\beta \neq 0$, then the matrix $\alpha E + \beta A$ has rank $1$. Therefore,
    $$
    \operatorname{rank}\begin{pmatrix}
        \alpha E + \beta A & B
    \end{pmatrix}
    = 2
    $$
    for every non-zero $(\alpha, \beta) \in \mathbb{C}^2$
    
    Finally,
    $$
    \operatorname{rank}\begin{pmatrix}
        F & C & D
    \end{pmatrix}
    =\operatorname{rank}\begin{pmatrix}
        0 & 1 & 0 & 0 &0 &0
    \end{pmatrix}
    = 1.
    $$
    Thus this Helmke system is controllable.
    
    We now show that the Helmke system is not $\rho$-stable. Let $L = \operatorname{span}(f_1) \subseteq V_1$, where $(f_1, f_2)$ is the standard basis of $V_1$. Since $\operatorname{im}B = L$, we have a subrepresentation
    $$
    U = (0, L, 0, \mathbb{C}^3, \mathbb{C}).
    $$
    Since the marked vertices are $v_0, v_1, v_2$, if follows that $\langle\rho, U\rangle_M = -2 \cdot 0 + 3\cdot 1 + 1\cdot 0 = 3 = \langle \chi, V\rangle_M$, so the subrepresentation is not $\chi$-stable. Since the Helmke system is controllable, Proposition \ref{Prop Controllability Boundary case} implies it is $\chi$-semistable. Hence the subrepresentation is strictly semistable and controllable.

    Moreover, the inclusion $\mathcal{H}^c_{2, 3, 1} \subseteq \mathcal{H}^{ss}_{2, 3, 1}$ is also strict. Recall the Helmke system $H_{ss}$ constructed in Lemma \ref{lemma 1}, where this Helmke system is shown to be strictly $\rho$-semistable, and in particular $H_{ss} \in \mathcal{H}^{ss}_{2, 3, 1}(\rho)$. We now check that $H_{ss}$ is not controllable. For every $(\alpha, \beta) \in \mathbb{C}^2$, we have
    $$
    \alpha E + \beta A = \begin{pmatrix}
    0 & 0\\
    \alpha & \beta
    \end{pmatrix},
    $$
    so the determinant $\det(\alpha E + \beta A) = 0$. Hence there is no pair $(\alpha, \beta)$ for which $\det (\alpha E + \beta A) \neq 0$. Thus $H_{ss} \in \mathcal{H}^{ss}_{2, 3, 1} \setminus H^c_{2, 3, 1}$, which proves the strict inclusion $\mathcal{H}^c_{2, 3, 1} \subsetneq \mathcal{H}^{ss}_{2, 3, 1}$.
\end{rmk}

Motivated by Proposition 1.22 in \cite{Bader2008}, we now introduce the notion of observability for Helmke systems. Recall that for a classical linear dynamical system $(A, B, C, D)$, unobservability means there exists $0 \neq U \subseteq \mathbb{C}^n$ such that $A(U) \subseteq U$, and $C(U) = 0$. In the quiver representing the Helmke system $(E, A , B, C, D, F)$, we have $v_0, v_1 \simeq \mathbb{C}^n$ with $E, A: v_0 \to v_1$. This means we need a pair of subspaces $U_0 \subseteq\mathbb{C}^n_{v_0}, U_1 \subseteq \mathbb{C}^n_{v_1}$ satisfying $E(U_0) + A(U_0) \subseteq U_1$ and $C(U_0) = 0$. The dimension condition is $\dim U_1 \leq \dim U_0$. We thus have the following definition

\begin{defn}\label{Helmke-Observability Defn}
    A Helmke system $H = (E, A, B, C, D, F)$ is \emph{observable} if there does not exist a pair $0\neq U_0\subseteq \mathbb{C}^n, U_1\subseteq\mathbb{C}^n$ such that
    $$
    E(U_0) + A(U_0) \subseteq U_1,\quad C(U_0) = 0,\quad \text{ and }\dim U_1 \leq \dim U_0.
    $$
\end{defn}

\begin{prop}
    The conditions in the above definition are equivalent to the nonexistence of a subrepresentation of the form
    $$
    (U_0, U_1, 0, 0, 0)
    $$
    with $0\neq U_0$ and $\dim U_1\leq \dim U_0$.
\end{prop}
\begin{proof}
    Suppose such a pair $(U_0, U_1)$ exists and $\dim U_0 > 0$. Then $\alpha E\mid_{U_0} +\beta A\mid_{U_0} : U_0 \to U_1$ is a linear map for some $(\alpha, \beta) \in \mathbb{C}^2$. If $\dim U_1< \dim U_0$, then the linear map has nontrivial kernel. If $\dim U_1 = \dim U_0$, then the determinant $\det(\alpha E\mid_{U_0} +\beta A\mid_{U_1})$ is a homogeneous polynomial of positive degree, hence has a zero over $\mathbb{C}$. Therefore, there exists some non-zero $(\alpha, \beta)$ and some non-zero $x \in U_0$ such that $(\alpha E + \beta A)x = 0$. Moreover, $C(U_0) = 0$ means that $Cx = 0$ for all $x \in U_0$, hence the matrix
    $$
    \begin{pmatrix}
        \alpha E + \beta A\\
        C
    \end{pmatrix} x = 0.
    $$
    Therefore, the rank
    $$
    \mathrm{rk}\begin{pmatrix}
        \alpha E + \beta A\\
        C
    \end{pmatrix} < n.
    $$

    Conversely, now let $\mathrm{rk}\begin{pmatrix}
        \alpha E + \beta A\\
        C
    \end{pmatrix} < n$ for some $(\alpha, \beta) \neq (0, 0)$, and choose a non-zero $x \in \ker(\alpha E + \beta A) \cap \ker C$. Let $U_0 = \mathbb{C}x$ and $U_1 = E(U_0) + A(U_0)$. Since $(\alpha E + \beta A)x = 0$, it follows that $Ex$ and $Ax$ are linearly dependent. Thus $\dim U_1 \leq 1 = \dim U_0$, which gives, together with $C(U_0)= 0$, the subrepresentations of the form $(U_0, U_1, 0, 0, 0)$.
\end{proof}

\begin{rmk}
Lemma 3.3 of \cite{Bader2008} states that there is a closed and equivariant embedding
$$
\phi_H: \Sigma_{n, m, p} \to \mathcal{H}_{n, m, p}
$$
given by $(A, B, C, D) \mapsto (I_n, A, B, C, D, I_p)$, where $\Sigma_{n, m, p}$ denotes the space of classical linear dynamical systems and $I_n, I_p$ denote the identity matrices. We observe that the conditions for Helmke observability become
$$
I_n(U_0) + A(U_0) = U_0 + A(U_0) \subseteq U_1,\quad C(U_0) = 0,\quad \text{ and }\dim U_1 \leq \dim U_0,
$$
which imply that $\dim U_1 = \dim U_0$ i.e. $U_0 = U_1$. Hence the conditions for Helmke observability coincide with the observability of classical linear dynamical systems in this case.

More generally, if matrices $E, F$ are invertible, i.e. $E\in \mathrm{GL}_n(\mathbb{C})$ and $F \in \mathrm{GL}_p(\mathbb{C})$, then the Helmke system can be rewritten as
$$
x_{k+1} = E^{-1}Ax_k + E^{-1}Bu_k,\quad y_k = F^{-1}Cx_k + F^{-1}Du_k.
$$
So for a Helmke system $H = (E, A, B, C, D, F)$ of type $(n, m, p)$, we have the corresponding classical linear dynamical system
\begin{equation*}
    \Sigma_H = (E^{-1}A, E^{-1}B, F^{-1}C, F^{-1}D). \qedhere
\end{equation*}
\end{rmk}

Now we want to interpret the notion of observability in terms of the defining equations of Helmke systems, which are
\begin{align}\label{Helmke Defining eqns}
    Ex_{k+1} = Ax_k + Bu_k,\quad Fy_k = Cx_k + Du_k,
\end{align}
where $u_k$ is the input, $y_k$ is the output, and $x_k$ is the state vector. Let us start by recalling the definition of observability for classical linear dynamical systems.
\begin{defn}
    A finite-dimensional dynamical system is called \emph{unobservable} over a time interval $[0, t]$ if distinct initial states cannot be distinguished by their input and output on $[0, t]$. A dynamical system is \emph{observable} if it is not unobservable.
\end{defn}

For Helmke systems, however, $E$ and $F$ may be singular. So a prescribed initial state and input sequence may not determine a unique state or output trajectory, i.e. $x_{k + 1}$ and $y_k$ need not be determined uniquely. Following the terminology of Polderman and Willems \cite{Polderman1998}, we therefore work with admissible trajectories, namely trajectories satisfying the defining equations of the systems.

\begin{defn}
    A state trajectory is called \emph{admissible} if it satisfies the defining equations throughout the entire time interval $[0, t]$. For an integer $T \in [0, t]$, an \emph{admissible triple} consists of $(x, u, y) =  \bigl( (x_k)_{k = 0}^T, (u_k)_{k = 0}^T, (y_k)_{k = 0}^T\bigr)$ satisfying equations (\ref{Helmke Defining eqns}).
\end{defn}

Let $(x^q,u, y^q)$ and $(x^0, u, y^0)$ be two admissible triples. Let $z_k := x^q_k - x^0_k$ and $w_k := y^q_k - y^0_k$. Subtracting the two equations (\ref{Helmke Defining eqns}) gives
$$
Ez_{k+1} = Az_k,\quad Fw_k = Cz_k.
$$
If the two pairs have the same output, then $w_k = 0$, and hence $Cz_k = 0$. Conversely, if a sequence $z = (z_k)$ satisfies $Ez_{k+1} = Az_k$ and $Cz_k = 0$, then $(z, 0, 0)$ is a zero output triple. Thus it can be added to any admissible triple to produce another admissible triple with the same input and the same output.

Therefore, we have the following definition
\begin{defn}\label{Helmke-observability Definition}
    A non-zero state $x^q_0 = q \in V_0$ of a Helmke system is called \emph{Helmke-unobservable} over the discrete time interval $\{0,\ldots, T\}$ if there exists a trajectory $z = (z_i)_{i = 0}^{T+1}$ with $z_0 = q$ such that $Ez_{k+1} = Az_k$ and $Cz_k = 0$ for $k = 0,...,T$. A Helmke system is called \emph{Helmke-observable} if it has no non-zero Helmke-unobservable states over $\{0,\ldots, T\}$.
\end{defn}

\begin{prop}\label{stability and  controllability and observability}
    Let $n, m, p \geq 1$, let $\mathcal{H}_{n, m, p}$ be the space of Helmke systems of type $(n, m, p)$, and let $\mathrm{G} = \mathrm{GL}_n(\mathbb{C}) \times \mathrm{GL}_n(\mathbb{C}) \times \mathrm{GL}_p(\mathbb{C})$. Let $\chi = (r, s, t) \in \mathbb{Z}^3$ be a character of $\mathrm{G}$, which is defined as
    $$
    \chi(g_0, g_1, g_2) = \det(g_0)^r\det(g_1)^s\det(g_2)^t.
    $$

    Assume $r + s = 0$, $r < 0$, $t > 0$, and $nr + (n-1)s + \min\{p, n\}t<0$. Then for a Helmke system $H = (E, A, B, C, D, F) \in \mathcal{H}_{n, m, p}$, the following are equivalent
    \begin{enumerate}
        \item $H$ is $\chi$-stable.
        \item $H$ is controllable and Helmke-observable.
    \end{enumerate}
\end{prop}
\begin{proof}
    Define
    $$
    \chi_N = (Nr, Ns + 1, Nt).
    $$
    Then $Nr + Ns+1 = N(r + s) + 1 = 1 > 0$, and $Nt > 0$ for all $N > 0$. Moreover,
    $$
    n\cdot Nr + (n-1)\cdot (Ns + 1) + \min\{p, n\}\cdot Nt = N(nr + (n-1)s + \min\{p, n\}t) + (n-1) < 0
    $$
    for sufficiently large $N$ because $nr + (n-1)s + \min\{p, n\}t < 0$. Consequently, $\chi_N$-stability is equivalent to controllability.

    Suppose $H$ is $\chi$-stable. Let $U = (U_0,U_1, U_2, U_3, U_4)$ be a subrepresentation that has either $U_3 = U_4 = 0$ or $U_3 = V_3$, $U_4 = V_4$. If $U_3 = U_4 = 0$, then $\chi$-stability gives $\langle \chi, U \rangle_M \geq 1$, which means $\langle\chi_N, U\rangle_M = N\langle\chi, U\rangle_M + \dim U_1 > 0$ for sufficiently large $N$. If $U_3 = V_3$ and $U_4 = V_4$, then $\chi$-stability gives $\langle\chi, U\rangle_M - \langle\chi, \mathrm{V}\rangle_M \geq 1$. Since $\langle \chi_N, U\rangle_M - \langle\chi_N, \mathrm{V}\rangle_M = N(\langle\chi, U\rangle_M - \langle\chi, \mathrm{V}\rangle_M) + (\dim U_1 - n)$, the right hand side is positive for sufficiently large $N$. Hence $H$ is $\chi_N$-stable when $N$ is sufficiently large, and thus controllable.

    Now show $H$ is Helmke-observable. If $H$ were not observable, then by Definition \ref{Helmke-Observability Defn}, there would exist a subrepresentation $(U_0, U_1, 0, 0, 0)$ with $0 \neq U_0$, $\dim U_1 \leq \dim U_0$. Since $r + s = 0$, we have $s = -r > 0$, and thus
    $$
    \langle \chi, U \rangle_M = r\dim U_0 + s\dim U_1 = r(\dim U_0 - \dim U_1) \leq 0,
    $$
    contradicting $\chi$-stability. Therefore, $H$ is Helmke-observable.

    Now suppose $H$ is controllable and Helmke-observable. Consider a non-zero subrepresentation
    $$
    U = (U_0, U_1, U_2, 0, 0).
    $$
    If $U_0 = 0$, then $\langle \chi, U \rangle_M = s\dim U_1 + t\dim U_2$, which is positive whenever $U \neq 0$. Now let $U_0 \neq 0$. Since $H$ is controllable, by definition we have $\det(\alpha E + \beta A) \neq 0$ for some $(\alpha, \beta)$. If $\dim U_1 < \dim U_0$, then every matrix $\alpha E + \beta A$ maps $U_0$ into $U_1$ of smaller dimension, which would fail the injectivity on $U_0$, contradicting the invertibility of some $\alpha E + \beta A$. Therefore one cannot have $\dim U_1 < \dim U_0$ for a subrepresentation with $E(U_0) + A(U_0) \subseteq U_1$. Thus $\dim U_1 \geq \dim U_0$.

    Note that
    $$
    \langle \chi, U\rangle_M = r\dim U_0 + s\dim U_1 + t\dim U_2 = s(\dim U_1 - \dim U_0) + t\dim U_2.
    $$
    If $\dim U_1 > \dim U_0$, this is positive. If $\dim U_1 = \dim U_0$ and $\dim U_2 >0$, this is also positive. If $\dim U_1 = \dim U_0$ and $\dim U_2 = 0$, then this would be of the form $(U_0, U_1, 0, 0, 0)$, which is a Helmke-unobservable subrepresentation by Definition \ref{Helmke-Observability Defn}. Hence
    $$
    \langle\chi, U\rangle_M > 0
    $$
    for every $U = (U_0, U_1, U_2, 0, 0)$.

    Now consider a subrepresentation
    $$
    U = (U_0, U_1, U_2, V_3, V_4).
    $$
    Because $H$ is controllable, it is $\chi_N$-stable for all sufficiently large $N$ \cite[Proposition 3.7]{Bader2008}. Therefore $\langle\chi_N, U\rangle_M - \langle \chi_N, \mathrm{V} \rangle_M > 0$. But
    \begin{align}\label{eqn}
        \langle\chi_N, U\rangle_M - \langle \chi_N, \mathrm{V} \rangle_M = N(\langle\chi, U\rangle_M - \langle\chi, \mathrm{V}\rangle_M) + (\dim U_1 - n).
    \end{align}
    Since $U_1 \subseteq V_1$, $\dim U_1 - n \leq 0$. If $\langle\chi, U\rangle_M - \langle\chi, \mathrm{V}\rangle_M = 0$, then the right hand side of \ref{eqn} becomes $\dim U_1 - n \leq 0$, contradicting $\chi_N$-stability. If $\langle\chi, U\rangle_M - \langle\chi, \mathrm{V}\rangle_M < 0$, then for sufficiently large $N$,
    \begin{align}\label{eqn2}
        N(\langle\chi, U\rangle_M - \langle\chi, \mathrm{V}\rangle_M) + (\dim U_1 - n) \leq -N + (\dim U_1 - n) \leq -N \leq 0,
    \end{align}
    which again contradicts $\chi_N$-stability. Hence $\langle\chi, U\rangle_M - \langle\chi, \mathrm{V}\rangle_M > 0$, i.e. $\langle\chi, U\rangle_M > \langle\chi, \mathrm{V}\rangle_M$.

    Combining both cases of subrepresentations, we have proved that $H$ is $\chi$-stable.
\end{proof}

\begin{rmk}
    It is worth giving a more explicit description on when $N$ is considered to be sufficiently large. The main idea is that we want the character $\chi_N$ to lie in the chamber given in Proposition 3.7 of \cite{Bader2008}. This is the case when
    $$
    N(nr + (n-1)s + \min\{p, n\}t) + (n - 1) < 0,
    $$
    which can be rearranged as
    $$
    N > \frac{n - 1}{-(nr + (n-1)s + \min\{p, n\}t)},
    $$
    where the denominator is positive by assumption.
    Moreover, note that the strict inequality in (\ref{eqn2}) also involves a choice of $N$, so it is convenient to choose an integer $N$ such that
    \begin{equation*}
        N > \mathrm{max}\bigl\{\frac{n - 1}{-(nr + (n-1)s + \min\{p, n\}t)}, n\bigr\}. \qedhere
    \end{equation*}
\end{rmk}
The preceding proposition can be viewed as the Helmke analogue of the zero character case as in Proposition 1.22 of \cite{Bader2008}. However, to obtain the Helmke analogue of the negative character case, restricting only $r + s < 0$ is not enough. The reason is that the Helmke quiver has three marked vertices $v_0, v_1, v_2$, and two unmarked vertices $v_3, v_4$. One should test two families of subrepresentations for stability \cite[Corollary 1.20]{Bader2008}, namely those with $U_3 = U_4 = 0$ and those with $U_3 = V_3$, $U_4 = V_4$. The first family contains the subrepresentations related to observability, because Helmke-observability is defined via the arrows
$$
E, A : v_0 \to v_1, \quad \text{ and }C: v_0 \to v_2.
$$
The second family, however, also involves the arrows $B : v_3 \to v_1$, $D: v_3 \to v_2$, and $F : v_4 \to v_2$.

\begin{prop}\label{Observable unmarked-zero}
    Let $H = (E, A, B, C, D, F)\in \mathcal{H}_{n, m, p}$, and let $\chi = (r, s, t) \in \mathbb{Z}^3$ satisfy $s, t > 0$, $r + s <0$, $r + t > 0$, and $(n-1)r + ns > 0$. Then the following are equivalent:
    \begin{enumerate}
        \item $H$ is Helmke-observable.
        \item For every non-zero subrepresentation
        $$
        U = (U_0, U_1, U_2, 0, 0),
        $$
        one has $\langle \chi, U \rangle_M > 0$.
        \item For every non-zero subrepresentation
        $$
        U = (U_0, U_1, U_2, 0, 0),
        $$
        one has $\langle \chi, U \rangle_M \geq 0$.
    \end{enumerate}
\end{prop}
\begin{proof}
    The implication $(2) \Rightarrow (3)$ is immediate. Now we prove $(3) \Rightarrow (1)$. Suppose $H$ is not Helmke-observable. Then there exist non-zero $(\alpha, \beta)$ and non-zero $x \in v_0$ such that $(\alpha E + \beta A)x = 0$, $Cx = 0$. Let $U_0 = \mathbb{C}x$, $U_2 = 0$, and $U_1 = \operatorname{span} \{Ex, Ax \}$, then
    $$
    E(U_0) + A(U_0) \subseteq U_1,\quad C(U_0) = 0.
    $$
    So $U = (U_0, U_1, 0, 0, 0)$ is a subrepresentation. Since $\alpha Ex + \beta Ax = 0$, the vectors $Ex$ and $Ax$ are linearly dependent. Hence $\dim U_1 \leq \dim U_0 = 1$. If $\dim U_1 = 0$, then $\langle \chi, U\rangle_M = r < 0$. If $\dim U_1 = 1$, then $\langle \chi, U\rangle_M = r + s < 0$.

    Now we prove $(1) \Rightarrow (2)$. Assume $H$ is Helmke-observable, and let $U = (U_0, U_1, U_2, 0, 0)$ be a non-zero subrepresentation. Then we have $E(U_0) + A(U_0)\subseteq U_1$ and $C(U_0) \subseteq U_2$. If $U_0 = 0$, then $\langle \chi, U\rangle_M = s\dim U_1 + t\dim U_2 > 0$. We may therefore suppose $U_0 \neq 0$. For each $(\alpha, \beta) \neq (0, 0)$, consider the map
    $$
    U_0 \to U_1 \oplus U_2,\quad x\mapsto ((\alpha E + \beta A)x, Cx).
    $$
    Helmke-observability implies that this map is injective for every non-zero $(\alpha, \beta)$, since if the map were not injective, then there would exist non-zero $x \in U_0$ such that $(\alpha E + \beta A)x =0$ and $Cx = 0$. Therefore
    $$
    \dim U_1 + \dim U_2 \geq \dim U_0.
    $$
    If $\dim U_1 + \dim U_2 = \dim U_0$ and $\dim U_1 > 0$, then after choosing bases of $(U_0, U_1, U_2)$, the determinant of the above map is a homogeneous polynomial of in $\alpha, \beta$ of positive degree $\dim U_1$. Note that over $\mathbb{C}$, every non-constant homogeneous polynomial in two variables has a zero on $\mathbb{P}^1$. This means there would exist non-zero $(\alpha, \beta)$ such that the determinant of the map vanishes, contradicting the injectivity. Thus for an observable system, either $\dim U_1 + \dim U_2 > \dim U_0$ or $\dim U_1 = 0$, $\dim U_2 = \dim U_0$.

    In the second case, we have
    $$
    \langle \chi, U \rangle_M = r\dim U_0 + t\dim U_0 = (r + t)\dim U_0 > 0
    $$
    since $r + t > 0$. In the first case, we have $\dim U_1 + \dim U_2 > \dim U_0$, thus $\dim U_0 \leq \dim U_1 + \dim U_2 - 1$. Because $r < 0$, the above gives
    $$
    r \dim U_0 + s\dim U_1 + t\dim U_2 \geq r(\dim U_1 + \dim U_2 - 1) + s\dim U_1 + t\dim U_2.
    $$
    Hence $\langle \chi, U \rangle_M \geq (r + s)\dim U_1 + (r + t)\dim U_2 - r$. Because $r + s < 0$, $r + t >0$, and $0 \leq \dim U_1 \leq n$, the right-hand side achieves its minimum when $\dim U_1 = n$ and $\dim U_2 = 0$. Therefore
    $$
    \langle \chi, U \rangle_M \geq (n-1)r + ns > 0.
    $$
    Hence we have established the Helmke analogue of the negative character part for subrepresentations with $U_3 = U_4 = 0$.
\end{proof}

\begin{rmk}
    The previous proposition only treats subrepresentations with $U_3 = U_4  = 0$. But full GIT stability for the Helmke quiver also tests the family where $U_3$ and $U_4$ have full dimension. These subrepresentations involve the arrows
    $$
    B : V_3 \to V_1,\quad D: V_3 \to V_2,\quad F : V_4 \to V_2,
    $$
    which are not involved in the definition of Helmke-observability. Thus one should not expect Helmke-observability to be equivalent to full $\chi$-stability on all of $\mathcal{H}_{n, m, p}$.

    For instance, consider the Helmke system of type $(2, 3, 1)$. Let
    $$
    E = \begin{pmatrix}
        0 & 1\\
        0 & 0
    \end{pmatrix},\quad
    A = \begin{pmatrix}
        0 & 0\\
        0 & 1
    \end{pmatrix},\quad
    C = \begin{pmatrix}
        1 & 0
    \end{pmatrix},
    $$
    and let $B, D, F$ be the zero map. We first check that this system is Helmke-observable. For any non-zero $(\alpha, \beta)$, we have
    $$
    \begin{pmatrix}
        \alpha E + \beta A\\
        C
    \end{pmatrix}
    \begin{pmatrix}
        x_1\\
        x_2
    \end{pmatrix}
    =
    \begin{pmatrix}
        \alpha x_2\\
        \beta x_2\\
        x_1
    \end{pmatrix}
    $$
    This vanishes when $x_1 = x_2 =0$. So this system is Helmke-observable.

    We now show that this Helmke-observable system is not $\chi$-semistable, and hence not $\chi$-stable. Consider the subrepresentation $U = (0,0,0, V_3, V_4)$. Since the marked vertices are $v_0, v_1$, and $v_2$, we have
    $$
    \langle \chi, U\rangle_M = 0 < \langle\chi, \mathrm{V}\rangle_M = 2r + 2s + t  = (r + 2s) + (r + t) > 0.
    $$
    Thus the system is not $\chi$-semistable and in particular, it is not $\chi$-stable.
\end{rmk}

\begin{prop}
    Let $H = (E, A, B, C, D, F) \in \mathcal{H}_{n, m, p}$, and suppose
    $$
    B(V_3) = V_1,\quad D(V_3) + F(V_4) = V_2.
    $$
    Let $\chi = (r, s, t) \in \mathbb{Z}^3$ satisfy $s, t > 0$, $r + s < 0$, $r + t > 0$, and $(n - 1)r + ns > 0$. Then the following are equivalent:
    \begin{enumerate}
        \item $H$ is Helmke-observable,
        \item $H$ is $\chi$-stable,
        \item $H$ is $\chi$-semistable.
    \end{enumerate}
\end{prop}
\begin{proof}
    The family of subrepresentations $U = (U_0, U_1, U_2, 0, 0)$ is treated in Proposition \ref{Observable unmarked-zero}. For the family of subrepresentations $U = (U_0, U_1, U_2, V_3, V_4)$, note that the assumptions $B(V_3) = V_1$ and $D(V_3) + F(V_4) = V_2$ give $U_1 = V_1$ and $U_2 = V_2$. Thus a proper subrepresentation in this family can only have $U_0 \subsetneq V_0$. Then
    $$
    \langle \chi, U \rangle_M - \langle\chi, \mathrm{V}\rangle_M = r(\dim U_0 - n) > 0,
    $$
    given that $r < 0$ and $\dim U_0 < n$. Hence the stability inequality for this family is automatic and strict.

    Combining this with Proposition \ref{Observable unmarked-zero} proves the equivalence of Helmke-observability, $\chi$-stability, and $\chi$-semistability.
\end{proof}
\begin{rmk}
    We now compare the chambers $r+s < 0$, $r + s = 0$, and $r + s >0$ in the above propositions with the classical ones in Proposition 1.22 of \cite{Bader2008}. In the classical case, the acting group is $\mathrm{GL}_n(\mathbb{C})$, and a character is determined by a single integer $\chi \in \mathbb{Z}$, namely
    $$
    g \longmapsto \det(g)^\chi.
    $$

    Bader's Proposition 1.22 shows that the sign of this single integer determines the control-theoretic meaning of GIT stability:
    \begin{itemize}
        \item if $\chi > 0$, then stability and semistability coincide and both are equivalent to controllability,
        \item if $\chi = 0$, then every classical system is semistable, and stability is equivalent to being both controllable and observable,
        \item if $\chi < 0$, then stability and semistability coincide and both are equivalent to observability.
    \end{itemize}
    Thus, for classical systems, the positive and negative sides are respectively the controllability and observability chambers, while the zero character is the point where both control-theoretic conditions appear simultaneously.

    For Helmke systems, the situation is similar but not identical. The group is now
    \begin{equation*}
        \mathrm{GL}_{n,n,p}=\mathrm{GL}_n(\mathbb C)\times \mathrm{GL}_n(\mathbb C)\times \mathrm{GL}_p(\mathbb C),
    \end{equation*}
    and a character has three components
    \begin{equation*}
        \chi = (r, s, t), \quad \chi(g_0, g_1, g_2) = \det(g_0)^r \det(g_1)^s\det(g_2)^t.
    \end{equation*}
    The two copies of $\mathrm{GL}_n$ act on the two state vertices $v_0$ and $v_1$, while the factor $\mathrm{GL}_p$ acts on the output vertex $v_2$. Therefore the analogue of the single classical character is not one of the three numbers $r, s, t$, but the sum $r + s$ when one restricts to the embedded classical locus. That is, under the classical embedding
    \begin{equation*}
        (A, B, C, D) \longmapsto (I_n, A, B, C, D, I_p),
    \end{equation*}
    the corresponding group homomorphism is
    \begin{equation*}
        \varphi_H : \mathrm{GL}_n(\mathbb C)\longrightarrow \mathrm{GL}_{n,n,p},\quad g \mapsto (g, g, I_p).
    \end{equation*}
    Pulling back the Helmke character along this homomorphism gives
    \begin{equation*}
        \chi \circ \varphi_H(g) = \det(g)^r\det(g)^s\det(I_p)^t = \det(g)^{r+s}.
    \end{equation*}
    So the cases $r+s < 0$, $r + s = 0$, and $r + s >0$ can be regarded as the Helmke analogue of the classical cases $\chi > 0$, $\chi = 0$, and $\chi < 0$.

    However, because the Helmke quiver has three marked vertices, the sign of $r + s$ alone does not determine the whole GIT chamber. For a subrepresentation $U = (U_0, U_1, U_2, U_3, U_4)$, the marked weight is
    \begin{equation*}
        \langle\chi, U\rangle_M = r\dim U_0 + s\dim U_1 + t\dim U_2.
    \end{equation*}
    The term $t \dim U_2$ has no direct analogue in the classical character. Moreover, by Corollary 1.20 of \cite{Bader2008}, one needs to test two different families of subrepresentations, namely those with unmarked vertices equal to zero, and those with the unmarked vertices full.

    The case $r + s > 0$ is the analogue of the positive classical character. In the chamber
    \begin{equation*}
        nr + (n-1)s + \min\{p, n\}t < 0, \quad r+s > 0, \quad t>0,
    \end{equation*}
    Bader's Proposition 3.7 states that, for Helmke systems, controllability, stability, and semistability all coincide. The boundary example considered above, $\rho = (-2, 3, 1)$, illustrates why the additional inequality is needed. Here $r + s = 1>0$ and $t = 1 > 0$, so the character $\rho$ still lies in the classical controllability chamber after restriction to the embedded classical locus. Nevertheless, $2r + s + t = 0$, so $\rho$ lies on the boundary of the Helmke controllability chamber. Consequently, Bader's chamber result no longer implies equality of the stable, semistable, and controllable loci. Instead, as shown in Proposition \ref{Prop Controllability Boundary case}, one obtains the inclusions
    $$
    \mathcal H^{st}_{n,m,p}(\chi) \subseteq \mathcal H^{c}_{n,m,p} \subseteq \mathcal H^{ss}_{n,m,p}(\chi),
    $$
    and strictly semistable Helmke systems can occur.

    The case $r + s = 0$ corresponds to the zero character case after restriction to the embedded classical locus. The reason is that the pullback $\chi \circ \varphi_H(g) = \det(g)^{r + s}$ is trivial when $r + s = 0$. However, it is not the same as taking the zero character on the Helmke space. The zero character for the Helmke action is $(0, 0 ,0)$, whereas a character satisfying $r + s = 0$ may still have $(r, s, t) \neq (0, 0 ,0)$. Therefore, unlike the classical zero character case, it is not true that every Helmke system is automatically semistable.
    
    Instead, the condition $r + s = 0$ is for both controllability and observability appear. The controllability part is obtained by perturbing the character slightly into the positive Helmke chamber, for instance by considering characters of the form $\chi_N = (Nr, Ns + 1, Nt)$, which satisfy $Nr + Ns + 1 > 0$ and, for $N$ sufficiently larger, lie in Bader's controllability chamber. The observability part comes from testing subrepresentations of the form $(U_0, U_1, 0, 0, 0)$, for which
    \begin{equation*}
        \langle \chi, U\rangle_M = r \dim U_0 + s \dim U_1 = r (\dim U_0 - \dim U_1).
    \end{equation*}
    Thus a subrepresentation with $\dim U_1 \leq \dim U_0$ is the type of subrepresentation that is Helmke-observable. In this sense, the case $r + s = 0$ is the Helmke analogue of the classical statement that when the character is zero, stability means controllable and observable. The difference is that, in the Helmke setting, the statement is formulated with the addition output parameter $t$ and with the marked quiver stability criterion.

    The case $r + s < 0$ is the analogue of the case where the character of classical systems is negative. Since the pullback of $\chi$ is then a negative power of the determinant, this is the chamber associated with observability. This parallels Bader's classical result that, for $\chi < 0$, stability and semistability are equivalent to observability. In the Helmke setting, however, the condition $r + s < 0$ itself does not control the full GIT stability problem. What it controls is the family of subrepresentations with $U_3 = U_4 = 0$, because this is the family where the arrows $E, A : v_0 \to v_1$, $C: v_0 \to v_2$ defines the Helmke-observability condition. The inequalities
    \begin{equation*}
        s,t>0,\qquad r+s<0,\qquad r+t>0,\qquad (n-1)r+ns>0
    \end{equation*}
    imposed in the preceding proposition ensure that the marked weight $\langle\chi, U\rangle_M$ is positive on all non-zero subrepresentations in this family when the Helmke system is observable.

    The full Helmke GIT problem also requires the family of subrepresentations with $U_3 = V_3$, $U_4 = V_4$. These subrepresentations involve the arrows
    \begin{equation*}
        B : V_3 \to V_1,\quad D : V_3 \to V_2,\quad F:V_4 \to V_2,
    \end{equation*}
    which are not part of the definition of Helmke-observability. Hence Helmke-observability by itself cannot be expected to imply full GIT stability or semistability on all of $\mathcal{H}_{n, m, p}$. This is the main new feature absent from the classical situation. Because any subrepresentation with $U_3 = V_3$ and $U_4 = V_4$ has $U_1 = V_1$ and $U_2 = V_2$, and the remaining inequality is ensured by the negativity of $r$, if we impose the additional constraints
    \begin{equation*}
        B(V_3) = V_1, \quad D(V_3) +F(V_4) = V_2,
    \end{equation*}
    then this family becomes automatic. Under these assumptions, the negative Helmke chamber becomes the analogue of the classical negative character statement, namely Helmke-observability, $\chi$-stability, and $\chi$-semistability coincide.

    In summary, the sign of $r + s$ can be interpreted as the classical character obtained by restricting the Helmke character to the embedded classical locus. But the Helmke chamber decomposition is richer than the classical structure, because Helmke systems have an additional marked output vertex and the two families of subrepresentations impose those extra inequalities.
\end{rmk}
\printbibliography
\end{document}